\documentclass[a4paper]{gtart}
\usepackage[inner=20mm, outer=30mm, textheight=245mm]{geometry}

\usepackage{amsmath,amsthm,amssymb,relsize}
\usepackage[nobysame,alphabetic, initials]{amsrefs}
\usepackage{mathrsfs,mathtools,makecell}
\usepackage[margin=1cm]{caption}

\usepackage{booktabs} % for pretty tables
\usepackage{siunitx} % for decimal alignment
\usepackage{extarrows}%for \xlongrightarrow

\usepackage{longtable}

\usepackage{booktabs, multirow, adjustbox, array}

\newcolumntype{R}[2]{%
    >{\adjustbox{angle=#1,lap=\width-(#2)}\bgroup}%
    l%
    <{\egroup}%
}

\usepackage{caption}

\usepackage{pifont}
\usepackage{floatrow}

\DeclareFloatFont{tiny}{\small}% "scriptsize" is defined by floatrow, "tiny" not
\usepackage{enumerate}
\usepackage{tikz}
\usetikzlibrary{
  cd,
  calc,
  positioning,
  fit,
  arrows,
  decorations.pathreplacing,
  decorations.markings,
  shapes.geometric,
  backgrounds,
  bending
}
\usepackage{tikzsymbols}
\PassOptionsToPackage{dvipsnames}{xcolor}
\definecolor{darkblue}{rgb}{0,0,0.6}

\usepackage[breaklinks, pdftex, ocgcolorlinks,colorlinks=true, citecolor=darkblue, filecolor=darkblue, linkcolor=darkblue, urlcolor=darkblue]{hyperref}
\usepackage[nameinlink,capitalize,noabbrev]{cleveref}

\usepackage[colorinlistoftodos]{todonotes}
\makeatletter
\newcommand{\mylabel}[2]{#2\def\@currentlabel{#2}\label{#1}}
\makeatother

\makeatletter
\newtheorem*{rep@theorem}{\rep@title}
\newcommand{\newreptheorem}[2]{%
\newenvironment{rep#1}[1]{%
 \def\rep@title{#2 \ref{##1}}%
 \begin{rep@theorem}}%
 {\end{rep@theorem}}}
\makeatother

\newtheorem{proposition}{Proposition}[section]
\newtheorem{theorem}[proposition]{Theorem}
\newtheorem{corollary}[proposition]{Corollary}
\newtheorem{lemma}[proposition]{Lemma}

\theoremstyle{definition}

\newtheorem{question}[proposition]{Question}
\newtheorem{example}[proposition]{Example}

\theoremstyle{remark}
\newtheorem{remark}[proposition]{Remark}

\newtheorem*{remark*}{Remark}

\newreptheorem{theorem}{Theorem}
\newreptheorem{lemma}{Lemma}
\newreptheorem{proposition}{Proposition}
\newreptheorem{corollary}{Corollary}
\newreptheorem{question}{Question}
\numberwithin{equation}{section}

\newcommand{\N}{\mathbb{N}}

\newcommand{\R}{\mathbb{R}}
\newcommand{\Z}{\mathbb{Z}}

\DeclareMathOperator{\vol}{vol}
\DeclareMathOperator{\volt}{volt}

\DeclareMathOperator{\volh}{vol_{\mathbb{H}}}

\usepackage{letltxmacro}

\LetLtxMacro\Oldfootnote\footnote

\begin{document}
\title{On a volume invariant of 3--manifolds}
\author{Marc Kegel, Arunima Ray, Jonathan Spreer, Em Thompson and Stephan Tillmann}

\begin{abstract}
This paper investigates a real-valued topological invariant of $3$--manifolds called \emph{topological volume}. For a given $3$--manifold $M$, the topological volume $\volt(M)$ is defined as the smallest volume of the complement of a (possibly empty) hyperbolic link in $M$. Various refinements of this invariant are given, asymptotically tight upper and lower bounds are determined, and all non-hyperbolic closed $3$--manifolds with topological volume of at most $3.07$ are classified. Moreover, it is shown that for all but finitely many cases, the volume minimiser in a lens space is obtained by Dehn filling one of the cusps of the complement of the Whitehead link or its sister manifold.
\end{abstract}



\makeshorttitle

%%%%%%%%%%%%%%%%%
%%%%%%%%%%%%%%%%%

\section{Introduction}

Let $M$ be a closed 3--manifold.
A topological invariant $\volt(M)$\footnote{The reader may interpret this as a quantification of how much potential the manifold has. We read this as \emph{topological volume} or \emph{volume topologique}. The name was coined with permission of Michel Boileau and Joan Porti.} is defined in~\cite{BBBMP} in an account of a complete proof of Thurston's Geometrisation Conjecture. Roughly speaking, this is the minimal volume of the complement of a (possibly empty) hyperbolic link in $M.$ Here we begin the study of this invariant in its own right by establishing some basic properties and computing it for certain families of non-hyperbolic 3--manifolds.

This introduces an interesting new organising principle for closed $3$--manifolds: Hyperbolic volume is one of the most widely used topological invariants for hyperbolic $3$--manifolds. By passing to hyperbolic links in (non-hyperbolic) manifolds we obtain a unified notion of volume for all $3$--manifolds. We show in \Cref{sec:Properties} that this volume invariant has interesting properties when studied over the whole class of closed 3--manifolds. Furthermore, we can partition the set of $3$--manifolds by asking for the number of components of the smallest volume hyperbolic link in $M$. Naturally, this number is zero if and only if $M$ is hyperbolic. It is an open question whether there exist manifolds for which this number is greater than one (cf. \Cref{que:unbounded_h(M)}).

As we demonstrate in this paper, topological volume is not just a theoretical concept. Instead, surprisingly powerful results about $\volt$ can be established by applying existing work on hyperbolic volume and exceptional surgeries. A key ingredient of our proofs is how hyperbolic volume behaves under Dehn filling according to work of J\o{}rgensen and Thurston \cite{thurston-notes,Gro81, Mai10}. This theory provides not just valuable tools, but also an insight into how hyperbolic volume is connected to the construction of general $3$--manifolds using Dehn surgery.

\paragraph{Related work}
A different but related \emph{link volume} was introduced by Rieck and Yamashita in~\cite{Rieck-link-2013}. A key difference between their invariant and $\volt$ is that there are at most finitely many 3--manifolds of the same topological volume (\Cref{pro:finiteness}) whilst there are examples of infinitely many 3--manifolds of the same link volume~\cite[Corollary 1.2]{Remigio-link-2012}. In this respect, our definition agrees with hyperbolic volume, since a given real number is realised as the volume of only finitely many hyperbolic 3--manifolds by J\o{}rgensen--Thurston theory \cite{thurston-notes,Gro81, Mai10}.

Another invariant that measures complexity of a closed orientable $3$-manifold $M$ is the Matveev complexity $c(M)$, which is defined as the
minimal number of vertices of a simple spine of $M$ \cite{Matveev}. It was shown that $c$ is additive under connected sum and agrees on irreducible manifolds that are not $S^3$, $\R P^3$, or $L(3,1)$ with the tetrahedral complexity $t(M)$. Our results show that $\volt$ behaves rather differently from $c$. For example we show in \Cref{sec:connected_sum} that $\volt$ is not additive under connected sum and in \Cref{ssec:upperlower} we observe that $\volt$ also behaves differently from the tetrahedral complexity $t$. 

On the other hand, we see a lot of overlap in our census of non-hyperbolic manifolds with topological volume at most $3.07$ with the census of non-hyperbolic manifolds of small Matveev complexity \cite{Matveev,complexity_census0,complexity_census}. For example all $19$ lens spaces with topological volume at most $3.07$ are a subset of the exactly $44$ lens spaces with Matveev complexity at most $5$, see Section 7 in \cite{complexity_census}.

\paragraph{Our contributions}

We begin, in \Cref{ssec:refinements}, by presenting some natural refinements of $\volt$. In particular, we consider $\volt_n(M)$ denoting the minimum over the volumes of all hyperbolic $n$-component links in $M$ and we define refinements of $\volt$ considering links representing a specified homology class. In the rest of \Cref{sec:Properties}, we establish some basic properties of $\volt$. In \Cref{sec:covers} we investigate how topological volume behaves under taking (branched) covers, and in \Cref{ssec:unbounded} identify sequences of closed 3--manifolds with unbounded topological volume, based on an observation on how topological volume interacts with classical topological invariants.

In \Cref{sec:bounds}, we strengthen these results by presenting upper and lower bounds for $\volt(M)$ for arbitrary closed $3$--manifolds $M$. In particular, the following two theorems can be proven as an application of results due to Adams \cite{Adams_Augmented}, Culler and Shalen \cite{Culler_Shalen2011,Culler_Shalen2012}, and Guzman and Shalen \cite{Guzman_Shalen,Guzman_Shalen2020}. 

\begin{theorem}\label{thm:surgery_bound}
  Let $M$ be a closed, orientable 3--manifold given by surgery along a nontrivial $n$-component link $L\subset S^3$ with crossing number $c=c(L)$. Then 
  \begin{equation*}
  	\volt(M)\leq \big(9c+15n-20\big)v_{oct}+4v_{tet}
  \end{equation*}
  where $v_{oct}\approx3.6638$ and $v_{tet}\approx1.01494$ are the volumes of a regular hyperbolic ideal octahedron and a regular hyperbolic ideal tetrahedron respectively.
\end{theorem}

\begin{theorem}\label{thm:lower_bounds}
    Let $M$ be a closed, orientable 3--manifold and $p$ be a prime. Then
    \begin{equation*}
        \frac{\operatorname{rk}(H_1(M;\Z_p))}{168.602}< \volt(M).
    \end{equation*}
    Moreover, we have $\volt(M)>3.08,$ $3.69$, $3.77$ for $\operatorname{rk}(H_1(M;\Z_2))>5$, $7$, $10$, respectively.
\end{theorem}

These bounds are asymptotically tight in the sense that there exist infinite families of $3$-manifolds, for which lower and upper bounds grow at the same rate up to a multiplicative constant: For instance, in \Cref{ex:SFS} we follow from \Cref{thm:surgery_bound,thm:lower_bounds} that for $M$ a Seifert fibered space over a genus $g$ orbifold with $n$ exceptional fibres of index $p$ a prime we have
\[ \frac{2g + n}{168.602} < \volt(M ) < 37g + 22n - 28.\] 

Dunfield's classification of exceptional Dehn fillings of cusped $3$--manifolds \cite{Du20} allows us to determine the exact topological volume of all non-hyperbolic closed $3$--manifolds of topological volume at most $3.07$. See \Cref{tab:complete} and \Cref{sec:computations} for summaries and discussion of observations arising from this computation.

The last result of this paper observes that for all but finitely many lens spaces, their volume minimiser arises as a Dehn filling of either the complement of the Whitehead link $W$ or its sister manifold $P$. Its proof exhibits the interesting interplay between the computation of topological volume of non-hyperbolic 3--manifolds and properties of the volume spectrum of cusped hyperbolic 3--manifolds.

\begin{theorem}\label{prop:lens-space-limit}
For all but finitely many homeomorphism types of lens spaces $L(p,q)$ ($p>q\geq 1$ coprime), the volume minimising link in $L(p,q)$ has complement one of $W(\frac{p'}{q'}, \infty)$ or $P(\frac{p''}{q''}, \infty)$, where the filling slopes $\frac{p'}{q'}$ and $\frac{p''}{q''}$ can be determined from an efficiently computable finite set. For $P$ this means that we fill along the unknotted component.
\end{theorem}

We hope that this survey over the properties of $\volt$, its refinements, and first results will convince the reader of the usefulness of this topological invariant as both an organising principle as well as a meaningful topological property. And in fact, our article has already inspired further work. In~\cite{schmalian} Schmalian proves that $\volt$ is algorithmically computable. More precisely he describes an algorithm that takes as input a triangulation of a $3$-manifold $M$ and a real number $\varepsilon>0$ and returns in finite time $\volt(M)$ up to an error of at most $\varepsilon$. Similar, results hold true for the refinements of $\volt$ that we discuss below. Moreover, in~\cite{volt2} $\volt$ is studied on product manifolds and computed for some manifolds with boundary. 

\textbf{Acknowledgements.} 
The authors thank the MATRIX Institute in Creswick and the Ma\-the\-ma\-tisches Forschungsinstitut Oberwolfach for a collaborative environment between different continents and time-zones during a tandem workshop in 2021. This is where this work was begun. They thank the organisers and participants of the workshop for stimulating conversations. Special mention goes to Jonathan Bowden, Jessica Purcell and Saul Schleimer for comments and contributions during the initial phase of this project. We also thank the anonymous referee for useful feedback. In particular, for providing the argument for the first statement of \Cref{cor:p-1}.

Figures~\ref{fig:m003}--\ref{fig:m034} were created using KLO~\cite{KLO}.

MK, AR, and ST thank the Max-Planck-Institut f\"ur Mathematik in Bonn, where parts of this work were carried out, for its hospitality. MK was supported by the SFB/TRR 191 \textit{Symplectic Structures in Geometry, Algebra and Dynamics}, funded by the DFG (Projektnummer 281071066 - TRR 191) and is now supported by a Ram\'on y Cajal grant (RYC2023-043251-I) and by project PID2024-157173NB-I00 funded by MCIN/AEI/10.13039/501100011033, ESF+ and FEDER, EU; and by a VII Plan Propio de Investigaci\'on y Transferencia (SOL2025-36103) of the University of Sevilla.
Research of JS and ST is supported in part under the Australian Research Council's Discovery funding scheme (project number DP190102259). JS thanks Michael Joswig and Technische Universit\"at Berlin, where parts of this work were carried out, for their hospitality. ET is supported by an Australian Government Research Training Program (RTP) Scholarship.

%%%%%%%%%%%%%%%%%%%%%%%%%%%%%%%%%%%%%%%%%%%%%%%%%%%%%

\section{Properties, examples and questions}
\label{sec:Properties}

We first define the notion of topological volume and refinements thereof. We then establish general properties of the topological volume, give examples and state questions that arise from our census in \Cref{sec:computations}.

\subsection{Definition of topological volume and its spectrum}\label{sec:definition}

Let $M$ be a closed, orientable 3--manifold.
A \textbf{link} in $M$ is a (possibly empty, possibly disconnected) closed 1--submanifold $L$ of $M.$ 
We consider links up to (not necessarily orientation preserving) diffeomorphisms of their complement $M\setminus L$. This is motivated by the fact that none of the invariants we consider in this article differ on non-isotopic links with diffeomorphic complement.

A link $L$ in $M$ is \textbf{hyperbolic} if its complement admits a complete hyperbolic structure of finite volume. In this case, let $\volh(M\setminus L)$ denote the hyperbolic volume of $M\setminus L.$ We often write $\volh(L) = \volh(M\setminus L).$
Following~\cite[\S 1.3.2]{BBBMP}, let
\begin{equation}
\volt(M) = \inf \big\{ \; \volh(M\setminus L) \;\mid \; L \subset M \text{ is a hyperbolic link } \; \big\}.
\end{equation}
It follows from Mostow--Prasad rigidity that $\volt(M)$ is a topological invariant of $M$, and we term it the \textbf{topological volume} of $M.$ 
The invariant $\volt(M)$ is indeed finite since every closed, orientable 3--manifold contains a hyperbolic knot~\cite{My82,My93}.\footnote{\cite{My93} in fact established that all compact $3$--manifolds, including non-orientable $3$--manifolds, contain hyperbolic knots. However we restrict ourselves to the closed and orientable case.} If $M$ admits a complete hyperbolic structure, then $\volt(M)=\volh(M)$, since volume decreases under Dehn filling.

\begin{example}\label{ex:volume-S3}
As our first example, we observe that $\volt(S^3)=2{v}_0 \approx 2.02988$, where ${v}_0$ is the volume of the ideal regular tetrahedron in hyperbolic $3$-space, and with the figure eight knot as unique minimiser. To see this, recall that \cite{cao-meyerhoff} showed that the complement of the figure eight knot, called $K4a1$ in the Hoste--Thistlethwaite--Weeks knot table \cite{HTWKnots}, and $m004$ in SnapPy's \cite{snappy} \texttt{OrientableCuspedCensus},\footnote{The naming convention throughout this paper follows that of the SnapPy census manifolds. Hence the symbol $m$ means that the corresponding manifolds can be triangulated with at most $5$ ideal tetrahedra.} is a cusped hyperbolic 3--manifold of minimum volume. Indeed, the cusped hyperbolic 3--manifolds of minimum volume are the once-punctured torus bundles $m003$ and $m004$ and they are related by a change in monodromy. However, $m003$ admits no filling to $S^3$ and thus does not represent a knot complement in $S^3$. All exceptional fillings of these two manifolds have the same topological volume.
\end{example}

According to J\o{}rgensen and Thurston (see~\cite{thurston-notes,Gro81, Mai10}), the set of volumes of complete orientable hyperbolic 3--manifolds of finite volume is a closed, non-discrete, well-ordered subset of the positive reals of order type $\omega^\omega.$ Throughout this article, we refer to this fact, as well as related results, as J\o{}rgensen--Thurston theory. In particular, there is always at least one hyperbolic link  $L \subset M$ with $\volt(M) = \volh(M\setminus L)$. We term such a link a \textbf{volume minimiser}. We say that two minimising links are equivalent if they have diffeomorphic complements. Occasionally we refer to the least volume hyperbolic link complement in $M$ as a volume minimiser of $M$ as well. 

\begin{question}
Let $M$ be a closed, orientable $3$--manifold $M$. Suppose $L_1$ and $L_2$ are links realising $\volt(M)$. Is $M \setminus L_1$ diffeomorphic with $M \setminus L_2$?
\end{question}

The manifold $m003$ appears as complement of a minimiser in $L(5,1)$ in two different ways---through the fillings $m003(-1, 1)$ and $m003(0, 1).$ However, there is a diffeomorphism of $m003$ that sends one surgery slope to the other. In particular, the minimising hyperbolic knot in $L(5,1)$ is the same. Hence this is not a counterexample to the following question, which is equivalent to the previous one if the cosmetic surgery conjecture is true.

\begin{question}
Let $M$ be a closed, orientable $3$--manifold $M$. Suppose $L_1$ and $L_2$ are links realising $\volt(M)$. Is there a diffeomorphism $M\to M$ taking  $L_1$ to $L_2$?
\end{question}

It is shown in~\cite{Adams_unbounded_volume} that the set of volumes of hyperbolic knots in any 3--manifold is unbounded. Hence not only determining the minimum but also the study of the spectrum of $\volt$ is interesting. 
Let $\mathcal{V}$ be the set of volumes of all cusped hyperbolic 3--manifolds, and $\mathcal{N}$ be the set of all closed, non-hyperbolic, orientable 3--manifolds. 
\begin{question}
Is $\volt\co \mathcal{N} \to \mathcal{V}$ surjective?
\end{question}

\subsection{Refinements of topological volume}
\label{ssec:refinements}

Let $M$ be a closed, orientable $3$--manifold. For each integer $n\ge 0,$ we define the invariant
\begin{equation}
\volt_n(M) = \inf \big\{ \; \volh(M\setminus K) \;\mid \; K \subset M \text{ is a hyperbolic $n$--component link } \; \big\}.
\end{equation}
For all $n\ge 1$ and all closed, orientable 3--manifolds $M$,
we have $\volt(M)\le \volt_n(M)$. 
In particular, $\volt(M)=\volt_0(M)< \volt_n(M)$ if and only if $M$ is hyperbolic. This also allows us to define an integer valued invariant
\[
h(M) = \min \{ \; n \; \mid \volt(M)= \volt_n(M)\;\}.
\]
Hence $M$ is hyperbolic if and only if $h(M)=0.$ The examples presented in \Cref{sec:computations} and based on the census in~\cite{GHMTY21} satisfy $h(M)=1.$

\begin{question}\label{que:unbounded_h(M)}
Is there an explicit family of examples $(M_n)$ with $h(M_n)=n$?
\end{question}

Myers~\cite{My93} showed that every homology class (indeed every free homotopy class) in a closed $3$--manifold is represented by some hyperbolic knot. More generally, he showed that every given finite set of free homotopy classes can be represented by a hyperbolic link. Hence for each class $c \in H_1(M),$ there is an invariant
\begin{equation}
\volt(M, c) = \inf \big\{ \; \volh(M\setminus L) \;\mid \; L \subset M \text{ is an oriented hyperbolic link with } [L]=c\; \big\}.
\end{equation}
A natural question is to ask which homology classes satisfy $\volt(M, c)= \volt(M).$
Some examples and a further discussion of this refinement can be found in \Cref{sec:Homology classes of minimisers}.

\subsection{A finiteness result}

According to J\o{}rgensen--Thurston theory \cite{thurston-notes,Gro81, Mai10}, for each positive real number $v$, there are at most finitely many hyperbolic 3--manifolds of volume $v.$ We show that the analogous result holds for topological volume.

\begin{proposition}\label{pro:finiteness}
For each positive real number $v$, there are at most finitely many closed, orientable 3--manifolds of topological volume $v.$
\end{proposition}

\begin{proof}
Suppose there are infinitely many closed 3--manifolds of topological volume $v$. Since only finitely many of them can be hyperbolic~\cite{thurston-notes}, this implies that there are infinitely many closed {\it non-hyperbolic} 3--manifolds $M_i$ of topological volume $v.$
In each $M_i$, choose a minimiser $L_i$. Then $\volh(M_i\setminus L_i) =v$, for each $i$, and hence there are only finitely many possibilities for $M_i\setminus L_i.$ For the rest of the proof suppose the cusped hyperbolic 3--manifold $M$ is a minimiser for infinitely many closed non-hyperbolic 3--manifolds $M_i$. In particular, for these manifolds, we have $M_i = M\big{(}\frac{p^i_1}{q^i_1}, \ldots, \frac{p^i_k}{q^i_k}\big{)}$. Since $M$ is a minimiser, leaving $k-1$ cusps unfilled and filling the $j$--th cusp with surgery coefficient $\frac{p^i_j}{q^i_j}$ gives a non-hyperbolic 3--manifold, since otherwise we would have a hyperbolic link complement of smaller volume than $v$ in $M_i.$ But only finitely many slopes $\frac{p^i_j}{q^i_j}$ on the $j$--th cusp may lead to a non-hyperbolic filling~\cite{thurston-notes}. Observing this for every cusp independently, we conclude that there are only finitely many possibilities for $M_i.$
\end{proof}

\subsection{Covers and branched covers}
\label{sec:covers}

It is straightforward to see that hyperbolic volume is multiplicative under finite-sheeted covers. The following example shows that the analogous fact is not true for topological volume in general. 

\begin{example}
We show in \cref{sec:computations} that $m009$ (Snappy notation) is the unique minimiser of $\R P^3,$ and hence 
$\volt(\R P^3)=\volh(m009)\approx 2.66674$. One sees this by noting that $\R P^3$ is obtained as a Dehn filling of $m009$ in a unique way, but it is not obtained by filling a cusped hyperbolic 3--manifold of smaller volume. As we saw in \cref{ex:volume-S3}, we have $\volt(S^3)\approx 2.02988$, so 
$$\volt(S^3) < 2\volt(\R P^3).$$
This shows that the topological volume is not multiplicative under finite covers in general. A surgery description of $m009$ is given by performing a $2$-surgery on one component of the Whitehead link $L5a1$ (using notation from the Thistlethwaite link table). The complement of the other link component in this surgery yields $m009$. From this surgery description one can see that the lift to $S^3$ of the minimiser in $\R P^3$  is the complement of $L6a1$ with volume $2 \cdot \volt (\mathbb{R}P^3) \approx 5.33$.
\end{example}

Nonetheless we have the following behaviour of topological volume under finite-sheeted covers. 

\begin{lemma}\label{lem:covers}
Let $X$ and $Y$ be closed, orientable 3--manifolds. If there is a $d$--fold cover $X\to  Y$, then 
\begin{equation}
\label{eq:covering}
\volt(X) \leq  d\cdot \volt(Y).
\end{equation}
\end{lemma}

\begin{proof}
Suppose $X$ and $Y$ are closed, orientable 3--manifolds with $p\co X \to Y$ a $d$--fold cover.
Suppose $L $ in $Y$ is a hyperbolic link. Then $X \setminus p^{-1}(L) \to Y \setminus L$ is a $d$--fold cover and in particular $p^{-1}(L) \subset X$ is a hyperbolic link, with $\volh (X \setminus p^{-1}(L)) =d \cdot  \volh(Y \setminus L)$. A minimiser $L \subset Y$  hence gives 
$\volt(X) \le \volh (X \setminus p^{-1}(L)) =d \cdot  \volh(Y \setminus L) = d \cdot  \volt(Y).$
\end{proof}

\begin{example}
For each $d>0,$ choose $p$ and $q$ such that $d \cdot p$ and $q$ are coprime. Then we have a $d$--fold covering of lens spaces 
$L(p, q) \to L(d\cdot p,q).$
 We show in \cref{prop:lens-space-limit} that for $p$ sufficiently large, $\volt(L(d\cdot p, q))$ and $\volt(L(p, q))$ are arbitrarily close to the volume of the Whitehead link complement. Hence the gap in \Cref{eq:covering} can be arbitrarily large, as $d$ varies.
\end{example}

\begin{question}
If $X\to  Y$ is a $d$--fold cover of non-hyperbolic closed, orientable 3--manifolds, is $\volt(X) <  d\cdot \volt(Y)$, that is, is Inequality~(\ref{eq:covering}) always strict?
\end{question}

As an application of \cref{lem:covers} we observe that knowledge about the topological volume can be leveraged to give bounds on degrees of covers. Recall that a link $L\subset S^3$ is called \emph{universal} if every closed, orientable 3--manifold arises as a cover of $S^3$ branched along $L$. The existence of universal links was first shown by Thurston in unpublished work. By work of Hilden, Lozano, and Montesinos~\cite{HLM2,HLM4,HLM3}, we know that there are numerous universal links, including the figure eight knot, the Whitehead link, and the Borromean rings.

\begin{corollary}
Let $L$ be a hyperbolic universal link in $S^3$ and let $d_L(M)$ denote the minimal degree of a cover $M\rightarrow S^3$ branched along $L$. Then
\begin{align*}
    \volt(M)\leq d_L(M) \volh(L).
\end{align*}
For example, we have
\[
\pushQED{\qed} 
 \volt(M)\leq d_{K4a1}(M) \volh(K4a1)=d_{K4a1}(M)\volt(S^3).\qedhere
\popQED
\]  
\end{corollary}

In particular, we obtain general lower bounds on $d_L(M)$ from the the topological volumes of $S^3$ and $M$. On the other hand, we remark that $d_L(M)$ is algorithmically computable, since the homeomorphism problem for $3$-manifolds is decidable: there exists a deterministic algorithm to decide if two given $3$-manifolds are homeomorphic or not~\cite{Kuperberg}. Thus we can enumerate all branched covers of $L$ ordered by the degree of the covering and check when the first such manifold is homeomorphic to $M$.

\subsection{Constructions of unbounded topological volume}
\label{ssec:unbounded}

The main observation used in this section is that, following J\o{}rgensen--Thurston theory \cite{thurston-notes,Gro81, Mai10}, all hyperbolic 3--manifolds of volume bounded from above by a constant are obtained by Dehn surgery on a finite set of cusped hyperbolic 3--manifolds. In particular, every infinite sequence of hyperbolic manifolds of bounded volume must have a subsequence with converging volumes, and which is obtained from a sequence of Dehn fillings on a fixed cusped hyperbolic $3$--manifold. This implies bounds on a number of topological invariants of those manifolds that have bounded topological volume, such as rank of the fundamental group, rank of homology, and hence number of JSJ pieces or prime summands in a decomposition of the manifold. In more precise terms we have the following result.

\begin{proposition}\label{prop:increasing-complexity}
	Let $\mathbf{I}$ be a real valued topological invariant of $3$--manifolds with the following property: If $M$ denotes a $3$--manifold with at least one torus boundary component $T$, then $\mathbf{I}(M)\geq \mathbf{I}(M(r))$ for all but finitely many slopes $r$, where $M(r)$ denotes the Dehn-filling of $M$ with slope $r$ along the boundary component $T$.
	Let $(M_n)_{n\in\N}$ be a family of closed, orientable $3$--manifolds with $\mathbf{I}(M_n) \rightarrow \infty$ for $n\rightarrow  \infty$. Then $\volt(M_n)\rightarrow \infty$ for $n\rightarrow \infty.$
\end{proposition}

\begin{proof}
Suppose that $(\volt(N_m))_{m\in \N}$ is a bounded sequence. Then, by J\o{}rgensen--Thurston theory, as explained at the beginning of \Cref {ssec:unbounded}, the elements of some subsequence $(N_{m_n})_{n\in\N}=(M_n)_{n\in\N}$ are obtained by Dehn filling a fixed cusped hyperbolic $3$--manifold $M$. This implies that, by assumption, the values $\mathbf{I} (M_{n})$ must also be bounded. This is a contradiction.
\end{proof}

\Cref{prop:increasing-complexity} has the following immediate consequences.

\begin{corollary}\label{cor:topinv}\hfill
	\begin{enumerate}
		\item Let $(M_n)$ be a sequence of closed, orientable $3$--manifolds with unbounded Heegaard genus or surgery number. Then $(\volt(M_n))$ is unbounded.
		\item Let $(M_n)$ be a sequence of closed, orientable $3$--manifolds with unbounded rank of first homology or fundamental group. Then $(\volt(M_n))$ is unbounded.
		\item For each $i$, let $M_i$ be a closed, orientable 3--manifold which is not homeomorphic to $S^3$. Then $\volt(M_1\#\cdots \#M_n) \to \infty$ as $n\to\infty$. 
		\item Let $(M_n)$ be a sequence of closed, orientable Seifert fibered spaces such that the number of exceptional fibres is unbounded. Then $(\volt(M_n))$ is unbounded. 
	\end{enumerate}
\end{corollary}

\begin{proof}
$(1)$ and $(2)$ follow directly from~\Cref{prop:increasing-complexity} since these complexities are well-known not to increase under Dehn filling.
$(3)$ follows from $(2)$ and the fact that the rank of the fundamental group is additive under connected sum. Similarly, the rank of the first homology is unbounded if the number of exceptional fibres in a Seifert fibered space is unbounded and thus $(4)$ follows.
\end{proof}

\section{Bounds on topological volume}
\label{sec:bounds}

Bounds on the topological volume of manifolds in terms of covers and branched covers were determined in \Cref{sec:covers}. 
We now give upper and lower bounds arising from topological or algebraic properties of a non-hyperbolic 3--manifold.

\subsection{Upper bounds from surgery descriptions or triangulations}
\label{ssec:upperlower}

Suppose the $3$--manifold $M$ is given by surgery on a hyperbolic link $L \subset S^3.$ By definition of the surgery diagram, there is a link $L^* \subset M$ such that $S^3 \setminus L = M\setminus L^*.$ We thus have $\volh (S^3 \setminus L) = \volh (M \setminus L^*) \geq \volt (M)$. 

\begin{example}
The 3--torus $T^3$ is given by $0$-surgery on all components of the Borromean rings $B$ (a hyperbolic link). Thus there is a hyperbolic link $B^* \subset T^3$ with the same complement as $B$ in $S^3.$ In particular, we have $\volt(T^3)\leq \volh(B)<7.33$. The topological volume of $T^3$ (and other product manifolds) is further studied in~\cite{volt2}. In particular, it is conjectured that $\volt(T^3)$ is realised by the Borromean rings and that there is no minimiser with fewer components. Thus $T^3$ is a candidate to have $h(T^3)>1$.
\end{example}

The following theorem expands on this idea to give a general upper bound.

\begin{reptheorem}{thm:surgery_bound}
    Let $M$ be a closed, orientable 3--manifold given by surgery along a nontrivial $n$-component link $L\subset S^3$ with crossing number $c=c(L)$. Then 
    \begin{equation*}
        \volt(M)\leq \big(9c+15n-20\big)v_{oct}+4v_{tet}
    \end{equation*}
    where $v_{oct}\approx3.6638$ and $v_{tet}\approx1.01494$ are the volumes of a regular hyperbolic ideal octahedron and a regular hyperbolic ideal tetrahedron respectively.
\end{reptheorem}

\begin{proof}
	The link $L$ may not be hyperbolic. We first show how to construct a hyperbolic link $\bar L$ in $S^3$, so that the complement of $L$ is a Dehn filling of the complement of $\bar L$. In other words, the complement of $\bar L$ is the complement of a hyperbolic link in $M$, so by computing an upper bound for the volume of $\bar L$ we obtain an upper bound on the topological volume of $M$. Note that in general $L$ will not be a sublink of $\bar L$.
	
	To construct $\bar L$ we follow an algorithm due to Adams~\cite{Adams_Augmented}. In that process we keep track of the new crossing number. Following this algorithm more carefully in specific examples will usually improve the stated bound in the theorem. We now describe the algorithm. 
	
	\begin{enumerate}
		\item We start with a diagram $D$ of $L$ that has the minimal number of crossings $c=c(L)$. First, we introduce a circle (an augmentation) around each twist region of the link diagram. This introduces $4$ new crossings for every twist region. Since the number of twist regions is bounded from above by the number of crossings, we obtain a link diagram $D_1$, with at most $5c$ crossings, consisting of the old diagram $D$ and augmentations $C_1$ around each twist region of $D$.
		
		\item We remove or change crossings in the twist regions of $D$ until we get a diagram $D_2$, consisting of a diagram of an alternating link $A_2$ together with circles $C_2=C_1$ around each twist region of $A_2$. This does not increase the number of crossings and thus $D_2$ has at most $5c$ crossings. Note that here $A_2$ is no longer isotopic to our original link $L$. However by cutting along disks bounded by the components of $C_1$ and twisting, we see that the complement of $L$ arises as a Dehn filling of the complement of $D_2$, which is all we need.
		
		\item We add unknotted components to $A_2$ (that introduce no new crossings with the augmentations) to transform it into a non-split alternating diagram. This will add at most $4(n-1)$ additional crossings to the link diagram $D_2$ and at most $2(n-1)$ twist regions to $A_2$, and produce a link diagram $D_3$ with at most $5c+4n-4$ crossings that consists of a non-split alternating diagram $A_3$ together with some augmentations $C_3=C_2$.
		
		\item For each reducing sphere of the diagram $A_3$, we add one unknotted component to the diagram to get a non-split, prime, alternating diagram $A_4$. This can be done in such a way that the new unknotted components have no crossings with the augmentation circles. The diagram $A_3$ has at most $c+4n-4$ crossings and thus it contains at most $(c+4n-4)/3$ reducing spheres. The unknotted components can be added in such a way that we introduce for each unknotted component $4$ crossings and two additional twist regions. Thus we get a diagram $D_4$ consisting of a non-split, prime, alternating diagram $A_4$ together with augmentations $C_4$ around all but at most $(2c+14n-14)/3$ twist regions. The diagram $D_4$ has at most $(19c+28n-28)/3$ crossings.
		
		\item We add augmentations around the at most $(2c+14n-14)/3$ twists regions of $A_4$ that are not already augmented. This yields a diagram $D_5$ of a link $\bar L$ that is augmented alternating and has at most $9c+15n-15$ crossings.
		
	\end{enumerate}
	The construction ensures that the link $\bar L$ is augmented alternating. It is thus hyperbolic by a result of Adams~\cite{Adams_Augmented}. Hence    \begin{equation*}
		\volt(M)\leq\volh(\bar L).
	\end{equation*}
	The claimed inequality in the theorem statement follows from another result of Adams~\cite{Adams_Bound}, based on a construction by Weeks~\cite[Section 3]{handbookknottheory}, stating that for every hyperbolic link $J$ in $S^3$ with more than $4$ crossings, the volume is bounded from above by $(c(J)-5)v_{oct}+4v_{tet}$.
\end{proof}

\begin{remark}
For a specific given surgery presentation, the bound above can be greatly improved. For example, in the case that a closed, orientable 3--manifold $M$ is described as surgery on a knot $K$, the main result of \cite{BAR-altvolume} implies that $\volt(M)\leq 10(5c(K)-1)v_{tet}$. The key ingredient of \cite{BAR-altvolume} is a result of \cite{Blair-altaug}, showing that the complement of every knot $K\subset S^3$ contains an unknot $U$ so that the link $K\sqcup U$ is hyperbolic. In \cite{BAR-altvolume} it is shown that the hyperbolic volume of the link $K\sqcup U$ is bounded above by $10(5c(K)-1)v_{tet}$. Since $S^3\smallsetminus (K\sqcup U)$ is naturally a submanifold of $M$, our claim follows.
\end{remark}

The above bound can be applied to the standard surgery diagram of a Seifert fibered space. However, in this specific situation the above strategy can be improved as follows.

\begin{proposition}\label{cor:SFS_bound}
    Let $SFS(g;r_1,\ldots,r_N)$ be the Seifert fibered space over a closed surface of genus $g$ with Seifert invariants $(r_1,\ldots,r_N)$. Then 
    \begin{equation*}
        \volt(SFS(g;r_1,\ldots,r_N))\leq \begin{cases}
        	v_{oct}<3.67 \, \textrm{ if $(g,N)=(0,0)$ or $(0,1)$}\\
(10g+6N-9)v_{oct}+4v_{tet}< 37g+22N-28 \,\textrm{ otherwise.}
        \end{cases}
    \end{equation*}
\end{proposition}

\begin{proof}
Consider the standard surgery diagram of the Seifert fibered space $SFS(g;r_1,\ldots,r_N)$ shown in~\cref{fig:SFS}, see for example Section 6.1 in \cite{GS}. The surgery link is a connected sum of $g$ copies of the Borromean rings and $N$ copies of the Hopf link. Add $N+g-1$ many further unknots to that surgery diagram as indicated in \cref{fig:SFS} to obtain a connected prime alternating diagram of a link $L$ that contains the surgery link of $SFS(g;r_1,\ldots,r_N)$ as a sublink. Since $L$ also contains the Borromean rings as a sublink, it is not a torus link. Now a result of Menasco~\cite{Menasco84} implies that $L$ is hyperbolic.
	
	Since $SFS(g;r_1,\ldots,r_N)$ is obtained by surgery on $L$, the volume of $L$ is an upper bound on the topological volume of $SFS(g;r_1,\ldots,r_N)$. From~\cref{fig:SFS} we see that $L$ has crossing number $10g+6N-4$ and thus the result of Adams~\cite{Adams_Bound} mentioned above implies the stated bound.
\end{proof}

\begin{figure}[htbp] 
	\centering
	\def\svgwidth{\columnwidth}
	\begingroup%
  \makeatletter%
  \providecommand\color[2][]{%
    \errmessage{(Inkscape) Color is used for the text in Inkscape, but the package 'color.sty' is not loaded}%
    \renewcommand\color[2][]{}%
  }%
  \providecommand\transparent[1]{%
    \errmessage{(Inkscape) Transparency is used (non-zero) for the text in Inkscape, but the package 'transparent.sty' is not loaded}%
    \renewcommand\transparent[1]{}%
  }%
  \providecommand\rotatebox[2]{#2}%
  \newcommand*\fsize{\dimexpr\f@size pt\relax}%
  \newcommand*\lineheight[1]{\fontsize{\fsize}{#1\fsize}\selectfont}%
  \ifx\svgwidth\undefined%
    \setlength{\unitlength}{532.72735392bp}%
    \ifx\svgscale\undefined%
      \relax%
    \else%
      \setlength{\unitlength}{\unitlength * \real{\svgscale}}%
    \fi%
  \else%
    \setlength{\unitlength}{\svgwidth}%
  \fi%
  \global\let\svgwidth\undefined%
  \global\let\svgscale\undefined%
  \makeatother%
  \begin{picture}(1,0.3030013)%
    \lineheight{1}%
    \setlength\tabcolsep{0pt}%
    \put(0,0){\includegraphics[width=\unitlength,page=1]{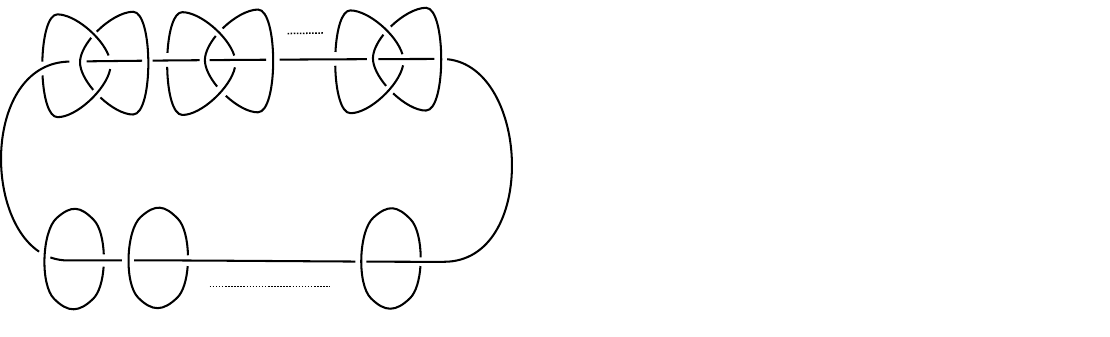}}%
    \put(0.46684641,0.15399914){\color[rgb]{0,0,0}\makebox(0,0)[lt]{\lineheight{1.25}\smash{\begin{tabular}[t]{l}$0$\end{tabular}}}}%
    \put(0.37604205,0.18033997){\color[rgb]{0,0,0}\makebox(0,0)[lt]{\lineheight{1.25}\smash{\begin{tabular}[t]{l}$0$\end{tabular}}}}%
    \put(0.22612052,0.17981216){\color[rgb]{0,0,0}\makebox(0,0)[lt]{\lineheight{1.25}\smash{\begin{tabular}[t]{l}$0$\end{tabular}}}}%
    \put(0.1598761,0.17875652){\color[rgb]{0,0,0}\makebox(0,0)[lt]{\lineheight{1.25}\smash{\begin{tabular}[t]{l}$0$\end{tabular}}}}%
    \put(0.11006829,0.17770089){\color[rgb]{0,0,0}\makebox(0,0)[lt]{\lineheight{1.25}\smash{\begin{tabular}[t]{l}$0$\end{tabular}}}}%
    \put(0.04183159,0.17717307){\color[rgb]{0,0,0}\makebox(0,0)[lt]{\lineheight{1.25}\smash{\begin{tabular}[t]{l}$0$\end{tabular}}}}%
    \put(0.05378548,0.00493048){\color[rgb]{0,0,0}\makebox(0,0)[lt]{\lineheight{1.25}\smash{\begin{tabular}[t]{l}$r_1$\end{tabular}}}}%
    \put(0.12941184,0.00334703){\color[rgb]{0,0,0}\makebox(0,0)[lt]{\lineheight{1.25}\smash{\begin{tabular}[t]{l}$r_2$\end{tabular}}}}%
    \put(0.26837566,0.28097825){\color[rgb]{0,0,0}\makebox(0,0)[lt]{\lineheight{1.25}\smash{\begin{tabular}[t]{l}$g$\end{tabular}}}}%
    \put(0.8095995,0.28742205){\color[rgb]{0,0,0}\makebox(0,0)[lt]{\lineheight{1.25}\smash{\begin{tabular}[t]{l}$g$\end{tabular}}}}%
    \put(0.2285294,0.01918151){\color[rgb]{0,0,0}\makebox(0,0)[lt]{\lineheight{1.25}\smash{\begin{tabular}[t]{l}$N$\end{tabular}}}}%
    \put(0.33810661,0.0054583){\color[rgb]{0,0,0}\makebox(0,0)[lt]{\lineheight{1.25}\smash{\begin{tabular}[t]{l}$r_N$\end{tabular}}}}%
    \put(0.31129189,0.17928434){\color[rgb]{0,0,0}\makebox(0,0)[lt]{\lineheight{1.25}\smash{\begin{tabular}[t]{l}$0$\end{tabular}}}}%
    \put(0,0){\includegraphics[width=\unitlength,page=2]{SFS.pdf}}%
    \put(0.76631286,0.02527104){\color[rgb]{0,0,0}\makebox(0,0)[lt]{\lineheight{1.25}\smash{\begin{tabular}[t]{l}$N$\end{tabular}}}}%
    \put(0,0){\includegraphics[width=\unitlength,page=3]{SFS.pdf}}%
  \end{picture}%
\endgroup%
	\caption{Left: The standard surgery diagram of a Seifert fibered space with genus $g$ and $N$ exceptional fibres. The surgery link is a connected sum of $g$ copies of the Borromean rings and $N$ copies of the Hopf link. Right: A hyperbolic link that contains the standard surgery diagram of a Seifert fibered space as a sublink.}
	\label{fig:SFS}
\end{figure}	

Although the bound from \cref{cor:SFS_bound} improves the general bound from \cref{thm:surgery_bound}, it is far from optimal. For example,  using SnapPy \cite{snappy} to compute the volume of $L$ for $g=2$ and $N=3$ gives a bound slightly smaller than $70$, while the upper bound from \Cref{cor:SFS_bound} is $112$.

\begin{proposition}
Let $M$ be a closed, orientable 3--manifold. Denote the Heegaard genus of $M$ by $g(M).$
There does not exist a function $f:~\mathbb{N} \to \mathbb{N}$ such that all 3--manifolds $M$ satisfy
        \begin{equation*}
            \volt(M)\leq f(g(M)).
        \end{equation*}
\end{proposition}

\begin{proof}
It suffices to demonstrate the existence of an infinite family of hyperbolic manifolds with unbounded volumes but bounded Heegaard genus. Indeed, there exist Berge knots with arbitrarily large volume~\cite[Theorem~4.1]{Baker_Berge_knots}. On the other hand, Berge knots have tunnel number one~\cite{Berge} and thus the knot exterior of a Berge knot has Heegaard genus $2$. Since Dehn surgery cannot increase the Heegaard genus, we obtain an infinite family of closed hyperbolic manifolds of Heegaard genus $2$ but arbitrarily large volume. 
\end{proof}

\begin{proposition}
    For a closed, orientable 3--manifold $M$, let $t(M)$ be the triangulation complexity of $M$, i.e. the minimal number of tetrahedra needed to triangulate $M$. 
 There exists an increasing function $f\colon \mathbb{N} \to \mathbb{N}$ such that for each $M$, we have
        \begin{equation*}
            \volt(M)\leq f(t(M)).
        \end{equation*}
\end{proposition}

\begin{proof}
This essentially follows from~\cite{My82}. For every triangulation of $M$, an explicit hyperbolic knot $K \subset M$ is constructed. From this construction we obtain an upper bound on the number of ideal tetrahedra needed to ideally triangulate the complement of $K$ in $M$. This bound only depends on the number of tetrahedra used to triangulate $M$, and not on the topology of $M$ or $M \setminus K$. The upper bound on the number of ideal tetrahedra gives an upper bound on the topological volume of $M$.
\end{proof}

\begin{remark}
	On the other hand, there is no function $f\colon \mathbb{N} \to \mathbb{N}$ that tends to infinity, such that 
	$$\volt(M) > f( t(M) ).$$
	This can be seen for example as follows. From \Cref{prop:lens-space-limit} we get that there exist infinitely many manifolds $M_n$ (lens spaces) with volt bounded by the volume of the Whitehead link. But since there are only finitely many manifolds with a given tetrahedral complexity $t$, it follows that $t(M_n)$ is unbounded. Thus such an inequality cannot exist. 
\end{remark}

\subsection{Lower bounds}

In \Cref{cor:topinv} (2) we have shown that a sequence of manifolds with unbounded rank of first homology has unbounded topological volume. In the next theorem, we develop a lower bound which quantifies this.

\begin{reptheorem}{thm:lower_bounds}
    Let $M$ be a closed, orientable 3--manifold and $p$ be a prime. Then
    \begin{equation*}
        \frac{\operatorname{rk}(H_1(M;\Z_p))}{168.602}< \volt(M).
    \end{equation*}
    Moreover, we have $\volt(M)>3.08,$ $3.69$, $3.77$ for $\operatorname{rk}(H_1(M;\Z_2))>5$, $7$, $10$, respectively.
\end{reptheorem}

\begin{proof}
   A result by Guzman--Shalen~\cite{Guzman_Shalen} states that for every finite-volume, orientable, hyperbolic 3-manifold $N$, and every prime $p$, we have $\operatorname{rk}(H_1(N;\Z_p)) < 168.602 \cdot  \volh(N)$. Now let $L \subset M$ be a volume minimiser of $M$, and set $N=M\setminus L$. Since the rank of the first homology cannot increase under filling, the first inequality follows, i.e.
   \begin{equation*}
       \frac{\operatorname{rk}(H_1(M;\Z_p))}{168.602}\leq\frac{\operatorname{rk}(H_1(M\setminus L;\Z_p))}{168.602}< \volh(M\setminus L)= \volt(M).
   \end{equation*}
   For the other inequalities we use the main results of~\cite{Culler_Shalen2011,Culler_Shalen2012,Guzman_Shalen2020} stating that for every \textit{closed}, finite-volume, orientable, hyperbolic $3$--manifold $N$ with $\operatorname{rk}(H_1(N;\Z_2))>5$, $7$ or $10$ we have $\volh(N)>3.08$, $3.69$ or $3.77$ respectively. Let $L$ be a volume minimiser for $M$. For every cusp $c_i$ of $M\setminus L$ we choose a sequence of slopes $p^{c_i}_n/q^{c_i}_n$ whose lengths converge to $\infty$ and such that filling $M\setminus L$ with these slopes yield manifolds $N_n$ with $\operatorname{rk}(H_1(N_n;\Z_2))=\operatorname{rk}(H_1(M;\Z_2))$. Note that the latter is satisfied if all $p^{c_i}_n$ are even. Since the lengths of the slopes converge to infinity it follows that {\em (a)} the manifolds $N_n$ are hyperbolic for $n$ sufficiently large, and {\em (b)} the volumes of $N_n$ converge to the volume of $M\setminus L$. Thus, applying the results of~\cite{Culler_Shalen2011,Culler_Shalen2012,Guzman_Shalen2020} to the $N_n$ yields the claimed inequalities for $\volt(M)$.
\end{proof}

\begin{example}\label{ex:SFS}
Although our general upper and lower bounds are far from sharp, we claim that their asymptotic behaviour is the same. Let $M$ be the Seifert fibered space of genus $g$ and $n$ singular fibres $SFS(g;p,\ldots, p)$ for a prime $p$. From the standard surgery description (see \cref{fig:SFS}) of $M$ we compute $\operatorname{rk}(H_1(M;\Z_p))=2g+n$ and thus \cref{thm:lower_bounds} and \cref{cor:SFS_bound} yield
    \begin{equation*}
        \frac{2g+n}{168.602} <\volt(M) < 37g+22n-28.
    \end{equation*}
    We observe that the quotient of the upper bound divided by the lower bound is bounded.
\end{example}

\section{The minimal topological volume manifolds}\label{sec:computations}

This section classifies all non-hyperbolic $3$--manifolds with topological volume less than $3.07$, giving their volumes and their volume minimising hyperbolic links. A number of observations from this census are subsequently made.

\subsection{A census}
\label{sec:census}

\begin{proposition}\label{prop:census}
The complete list of non-hyperbolic closed 3--manifolds of topological volumes at most $3.07$, their explicit volumes and their minimisers are given in \Cref{tab:complete}. In particular, each of these manifolds contains a unique volume minimiser. In all cases the volume minimiser is a knot.
\end{proposition}

\begin{proof}
From Theorem 1.5 of~\cite{GHMTY21}, cf.~\cite{Mi09}, we know the complete list of cusped hyperbolic 3--manifolds of volume at most $3.07$. These are:
$$\mathcal{V}=\big[m003,m004,m006,m007,m009,m010,m011,m015,m016,m017, m019, m022, m023, m026\big]$$
All of these have exactly one cusp. Hence the non-hyperbolic 3--manifolds of topological volumes at most $3.07$ are precisely the exceptional fillings on these manifolds, and thus there are finitely many.

We claim that the volumes of the manifolds in the above list satisfy:
\begin{align*}
\volh(m003) = \volh(m004) &< \volh(m006) = \volh(m007) \\
&< \volh(m009) = \volh(m010)\\
&< \volh(m011)\\
&< \volh(m015) = \volh(m016)= \volh(m017)\\
&< \volh(m019)\\
&< \volh(m022) = \volh(m023)\\
&< \volh(m026)
\end{align*}
We first note that the above equalities and inequalities are correct if one computes
the volumes with precision \emph{up to two decimal places}. Hence it remains to show that the stated equalities hold. 

We first note that $m003$ and $m004$ are once-punctured torus bundles with monodromies $-RL$ and $RL$, respectively. 
Here $R,L \in \mathrm{SL}(2,\mathbb{Z})$ denote the standard generators
\[
R=\begin{pmatrix}1&1\\0&1\end{pmatrix}, \qquad 
L=\begin{pmatrix}1&0\\1&1\end{pmatrix},
\]
which correspond to right and left Dehn twists about the standard generators of the once-punctured torus. 
Hence they are related by volume-preserving mutation along the fibre and so $\volh(m003)=\volh(m004)$.
Similarly, $m009$ and $m010$ are once-punctured torus bundles with monodromies $RRL$ and $-RRL$, respectively; and $m022$ and $m023$ are once-punctured torus bundles with monodromies $-RRRL$ and $RRRL$, respectively. 
We can verify these homeomorphisms using SnapPy by constructing the once-punctured torus bundles with the given monodromies and checking that they are isometric to the claimed census manifolds.

For $\volh(m006) = \volh(m007)$ and $\volh(m015) = \volh(m016)= \volh(m017)$ we refer to \Cref{app:shapes}, where we use the gluing and completeness equations to show that the manifolds claimed to have the same volume are built from tetrahedra with equal shapes.

In \cite{Du20}, Dunfield provides a complete list of all exceptional fillings of SnapPy \cite{snappy} manifolds with one cusp. This algorithmic search is made feasible by the $6$-theorem of Agol and Lackenby~\cite{agol-bounds,lackenby-6}, and the fact that the low complexity non-hyperbolic closed 3--manifolds that arise can easily be identified.
 
Thus we can determine the topological volume of any exceptional filling along a manifold in $\mathcal V$ as follows. First enumerate all exceptional fillings of the two manifolds of smallest volume. Then enumerate all exceptional fillings of the manifolds with second smallest volume in the list and only keep those non-hyperbolic 3--manifolds that have not yet been found. Then continue this process iteratively. If an exceptional filling is found as fillings of multiple distinct cusped manifolds of the same volume, or by filling one manifold with different surgery coefficients, we report all instances in which it arises. We present the results in \Cref{tab:complete}, where we sort our results by volume. 
\end{proof}

\begin{center}
\begin{longtable}{llll}
\caption{The complete list of non-hyperbolic manifolds of volume less than $3.07$ ordered by their volumes. Here the column \textit{Manifold} contains the Regina name~\cite{Regina} as used by Dunfield \cite{Du20}, where $S2$ $T$, $A$, and $D$ describe the case that the base surface is the $2$-sphere, the $2$-torus, the annulus, and the $2$-disk, respectively. Note that the notation of the Seifert fibered space agrees with the one used in \Cref{cor:SFS_bound} as for example shown in Section 7 of \cite{QA}. The column \textit{Volume} contains a verified approximation of the volume (i.e. at most the last digit is wrong), and the column \textit{Realisations} describes different ways to get the manifold by Dehn filling a SnapPy census manifold. Here and in the rest of this article the slopes in a cusp of a census manifold are measured with respect to the
	\textit{geometric basis} of the cusp given by the two shortest curves. This might in general differ from the Seifert basis if the census manifold is the complement of a knot in $S^3$.
	 \label{tab:complete}} \\
  
    \hline
    Manifold & Volume & Realisations &\\
    \hline
    \hline
    \endfirsthead

    \multicolumn{4}{l}%
    { \tablename\ \thetable{} -- continued from previous page} \\
    \hline
    Manifold & Volume & Realisations &\\
    \endhead

    \hline \multicolumn{4}{r}{{continued on next page --}} \\
    \endfoot

    \hline
    \endlastfoot

		$S^3$ & $2.02988321281931$ & $m004(1, 0)$ & \\
		$L(5,1)$ & $2.02988321281931$ & $m003(-1, 1)$ & $m003(0, 1)$\\
		$L(10,3)$ & $2.02988321281931$ & $m003(1, 0)$ &  \\
		$SFS [S2: (2,1) (3,2) (3,-1)]$& $$2.02988321281931$$ & $m003(-2, 1)$ & $m003(1, 1)$   \\
		$SFS [S2: (2,1) (3,1) (7,-6)]$& $2.02988321281931$ & $m004(-1, 1)$ & $m004(1, 1)$    \\
		$SFS [S2: (2,1) (4,1) (5,-4)]$& $2.02988321281931$ & $m004(-2, 1)$ &  $m004(2, 1)$  \\
		$SFS [S2: (3,1) (3,1) (4,-3)]$& $2.02988321281931$ & $m004(-3, 1)$ & $m004(3, 1)$   \\
		$T \times I / [ -2,-1 | -1,-1 ]$ & $2.02988321281931$ & $m003(-1, 2)$ & \\
		$T \times I / [ 2,1 | 1,1 ]$ & $2.02988321281931$ & $m004(0, 1)$ &  \\
		$SFS [D: (2,1) (2,1)] \cup_{(0,1 | 1,1)} SFS [D: (2,1) (3,2)]$ & $2.02988321281931$ & $m003(1, 2)$ & $m003(-3, 2)$\\
		$SFS [D: (2,1) (2,1)] \cup_{( 0,1 | 1,0)} SFS [D: (2,1) (3,1)]$ & $2.02988321281931$ & $m004(-4, 1)$ & $m004(4, 1)$ \\
		\midrule
		$L(15,4)$ & $2.56897060093671$ & $m006(0, 1)$ & \\
		$L(5,2)$ & $2.56897060093671$ & $m006(1, 0)$ & \\
		$L(3,1)$ & $2.56897060093671$ & $m007(1, 0)$ & \\
		$SFS [S2: (2,1) (3,2) (4,-3)]$ & $2.56897060093671$ & $m006(-1, 1)$ &  \\
		$SFS [S2: (2,1) (2,1) (3,2)]$ & $2.56897060093671$ & $m006(1, 1)$ &  \\
		$SFS [S2: (2,1) (2,1) (5,-2)]$ & $2.56897060093671$ & $m007(1, 1)$ &  \\
		$SFS [S2: (2,1) (3,1) (7,-5)]$& $2.56897060093671$ & $m006(-2, 1)$ &    \\
		$SFS [S2: (3,1) (3,1) (3,-1)]$& $2.56897060093671$ & $m007(0, 1)$ &    \\
		$SFS [S2: (2,1) (4,1) (5,-3)]$& $2.56897060093671$ & $m007(-1, 1)$ &    \\
		$SFS [S2: (2,1) (3,1) (9,-7)]$& $2.56897060093671$ & $m007(-2, 1)$ &    \\
		$SFS [A: (2,1)] / [ 0,-1 | -1,0 ]$ & $2.56897060093671$ & $m006(-3, 1)$ & \\
		$SFS [D: (2,1) (3,1)] \cup_{(-1,1 | 0,1)} SFS [D: (2,1) (3,1)]$ & $2.56897060093671$ & $m006(1, 2)$ & \\
		$SFS [A: (2,1)] / [ 0,1 | 1,0 ]$ & $2.56897060093671$ & $m007(-3, 1)$ & \\
		\midrule
		$L(2,1)$ & $2.66674478344906$ & $m009(1, 0)$ & \\
		$L(6,1)$ & $2.66674478344906$ & $m010(1, 0)$ & \\
		$SFS [S2: (2,1) (2,1) (4,-1)]$ & $2.66674478344906$ & $m010(0, 1)$ &  \\
		$SFS [S2: (2,1) (3,1) (8,-7)]$& $2.66674478344906$ & $m009(0, 1)$ &    \\
		$SFS [S2: (2,1) (4,1) (6,-5)]$& $2.66674478344906$ & $m009(-1, 1)$ &    \\
		$SFS [S2: (3,1) (3,1) (5,-4)]$& $2.66674478344906$ & $m009(-2, 1)$ &    \\
		$SFS [S2: (3,1) (3,2) (3,-1)]$& $2.66674478344906$ & $m010(1, 1)$ &    \\
		$T \times I / [ 3,2 | 1,1 ]$ & $2.66674478344906$ & $m009(1, 1)$ & \\
		$T \times I / [ -3,-2 | -1,-1 ]$ & $2.66674478344906$ & $m010(-2, 1)$ & \\
		$SFS [D: (2,1) (2,1)] \cup_{(0,1 | 1,0)} SFS [D: (2,1) (4,1)]$ & $2.66674478344906$ & $m009(-3, 1)$ & \\
		$SFS [D: (2,1) (2,1)] \cup_{(-1,2 | 0,1)} SFS [D: (2,1) (3,1)]$ & $2.66674478344906$ & $m009(3, 1)$ & \\
		$SFS [D: (2,1) (2,1)] \cup_{(0,1 | 1,1)} SFS [D: (2,1) (4,3)]$ & $2.66674478344906$ & $m010(2, 1)$ & \\
		$L(2,1) \# L(3,1)$ & $2.66674478344906$ &  $m010(-1, 1)$&\\
				\midrule
		$L(9,2)$ & $2.78183391239608$ & $m011(0, 1)$ & \\
		$L(13,3)$ & $2.78183391239608$ & $m011(1, 0)$ & \\
		$SFS [S2: (2,1) (3,2) (4,-1)]$ & $2.78183391239608$ & $m011(-1, 1)$ &  \\
		$SFS [S2: (2,1) (2,1) (3,-2)]$ & $2.78183391239608$ & $m011(1, 1)$ &  \\
		$SFS [S2: (2,1) (3,2) (5,-3)]$ & $2.78183391239608$ & $m011(2, 1)$ &  \\
		$SFS [D: (2,1) (3,1)] \cup_{(0,1 | 1,1)} SFS [D: (2,1) (3,2)]$ & $2.78183391239608$ & $m011(-1, 2)$ & \\
		\midrule
		$L(19,7)$ & $2.82812208833078$ & $m016(-1, 1)$ & \\
		$L(18,5)$ & $2.82812208833078$ & $m016(0, 1)$ & \\
		$L(21,8)$ & $2.82812208833078$ & $m017(-1, 1)$ & \\
		$L(14,3)$ & $2.82812208833078$ & $m017(0, 1)$ & \\
		$L(7,2)$ & $2.82812208833078$ & $m017(1, 0)$ & \\
		$SFS [S2: (2,1) (4,1) (7,-5)]$& $2.82812208833078$ & $m015(0, 1)$ &    \\
		$SFS [S2: (3,1) (3,1) (5,-3)]$& $2.82812208833078$ & $m015(1, 1)$ &    \\
		$SFS [S2: (2,1) (3,1) (11,-9)]$& $2.82812208833078$ & $m015(-1, 1)$ &    \\
		$SFS [D: (2,1) (2,1)] \cup_{(0,1 | 1,0)} SFS [D: (2,1) (3,2)]$ & $2.82812208833078$ & $m015(2, 1)$ & \\
		$SFS [A: (2,1)] / [ 0,1 | 1,-1 ]$ & $2.82812208833078$& $m015(-2, 1)$ & \\
		$SFS [D: (2,1) (2,1)] \cup_{(1,1 | 0,1)} SFS [D: (2,1) (3,1)]$ & $2.82812208833078$ & $m016(2, 1)$ & \\
		$SFS [D: (2,1) (3,1)] \cup_{(-1,1 | 0,1)} SFS [D: (2,1) (3,2)]$ & $2.82812208833078$ & $m016(-1, 2)$ & \\
		$SFS [D: (2,1) (2,1)] \cup_{(0,1 | 1,1)} SFS [D: (2,1) (3,1)]$ & $2.82812208833078$ & $m017(-2, 1)$ & \\
		$SFS [A: (2,1)] / ( 0,-1 | -1,1 )$ & $2.82812208833078$ & $m017(2, 1)$ & \\
		\midrule
		$L(17,5)$ & $ 2.944106486676$ & $m019(0, 1)$ & \\
		$L(11,3)$ & $ 2.944106486676$ & $m019(1, 1)$ & \\
		$SFS [S2: (2,1) (3,2) (5,-2)]	$ & $2.944106486676$ & $m019(-1, 1)$ &  \\
		$SFS [D: (2,1) (3,2)] \cup_{(0,1 | 1,1)} SFS [D: (2,1) (3,2)] ]	$ & $2.944106486676$ & $m019(-2, 1)$ &  \\
			\midrule
		$L(7,1)$ & $2.98912028293$ & $m022(1, 0)$ & \\
		$SFS [S2: (2,1) (3,1) (4,-1)]	$ & $2.98912028293$ & $m022(0, 1)$ &  \\
		$SFS [S2: (3,2) (3,2) (4,-3)]	$ & $2.98912028293$ & $m022(-1, 1)$ &  \\
		$SFS [S2: (2,1) (3,1) (9,-8)]	$ & $2.98912028293$ & $m023(0, 1)$ &  \\
		$SFS [S2: (2,1) (4,1) (7,-6)]	$ & $2.98912028293$ & $m023(1, 1)$ &  \\
		$SFS [S2: (3,1) (3,1) (6,-5)]	$ & $2.98912028293$ & $m023(2, 1)$ &  \\
		$T x I / [ -4,-3 | -1,-1 ]	$ & $2.98912028293$ & $m022(2, 1)$ &  \\
		$T x I / [ 4,3 | 1,1 ]	$ & $2.98912028293$ & $m023(-1, 1)$ &  \\
		$SFS [D: (2,1) (2,1)] \cup_{(0,1 | 1,1)} SFS [D: (2,1) (5,4)] ]	$ & $2.98912028293$ & $m022(-2, 1)	$ &  \\
		$SFS [D: (2,1) (2,1)] \cup_{(0,1 | 1,0)} SFS [D: (2,1) (5,1)]$ & $2.98912028293$ & $m023(3, 1)$ &  \\
			\midrule
		$L(19,4)$ & $3.059338057779$ & $m026(0, 1)$ & \\
		$L(8,3)$ & $3.059338057779$ & $m026(1, 0)$ & \\
		$SFS [S2: (2,1) (3,2) (5,-4)]	$ & $3.059338057779$ & $m026(-1, 1)	$ &  \\
		$SFS [S2: (2,1) (3,1) (3,2)]	$ & $3.059338057779$ & $m026(1, 1)$ &  \\
		$SFS [D: (2,1) (3,1)] \cup_{(-1,1 | 0,1)} SFS [D: (2,1) (4,1)]$ & $3.059338057779$ & $m026(1, 2)$ &  \\
		\bottomrule
\end{longtable}
\end{center}

\begin{remark}
Upper bounds on the volume of non-hyperbolic manifolds not appearing in our census can be computed from Dunfield's data~\cite{Du20}. The sharpness on these bounds depends on determining further terms in the volume spectrum of one-cusped hyperbolic 3--manifolds. The data ordered for our setting can be accessed at~\cite{data}.
\end{remark}

\subsection{Parenthood}

Next, we describe all ways in which the fourteen $1$-cusped manifolds in $\mathcal{V}$ of volumes at most $3.07$ can arise as fillings of the two $2$-cusped manifolds of minimal volume~\cite{agol-2cusped}. These are the Whitehead link complement $m129$ and the Whitehead link sister $m125$; compare~\Cref{sec:WandP}.

\begin{proposition}
	\label{prop:isometry}
	Every manifold in $\mathcal{V}$ can be obtained by filling the first cusp of $m125$ or $m129$. The complete list of surgeries on $m125$ and $m129$ that yield a manifold in $\mathcal{V}$ is shown in~\cref{tab:parent}. Because $m125$ has a symmetry that preserves the cusps but acts as $\bigl(\begin{smallmatrix}0 & -1 \\ 1 & 0\end{smallmatrix}\bigr)$ on the geometric basis, we obtain each manifold that arises as filling of $m125$ in two ways. Both manifolds admit an isometry that interchanges the components and thus we also have the analogous result for the second cusp.
\end{proposition}

\begin{proof}
	For the proof we use a result of Futer--Kalfagianni--Purcell~\cite{FuterKalfagianniPurcell} giving a bound on the possible length of slopes yielding manifolds of volume at most $3.07$. In~\cite{data} we explicitly compute these finite sets of slopes and identify all filled manifolds. The results are summarised in~\cref{tab:parent}.
\end{proof}

\begin{table}[htbp]
	\caption{The parenthood diagram of the low volume $1$-cusped manifolds. \label{tab:whichone}}
	\label{tab:parent}
	\begin{tabular}{lll}
		Manifold & Surgery on $m129$ &Surgery on $m125$ \\
		\toprule	
		$m003$ & $m129(3, 1)  (0,0)$ & $m125(-2, 1)(0,0)$ and $m125(1, 2)(0,0)$\\ 
		
		$m004$ & $m129(-3, 1) (0,0)$ & none \\  
		
		$m006$ & $m129(1, 2) (0,0)$ & $m125(3, 1)(0,0)$ and $m125(-1, 3)(0,0)$\\  
		
		$m007$ & $m129(-1, 2) (0,0)$ & none \\ 
		
		$m009$ & $m129(-4, 1) (0,0)$ & none \\ 
		
		$m010$ & $m129(4, 1) (0,0)$ & none \\ 
		
		$m011$ & none & $m125(-3, 1) (0,0)$ and $m125(1, 3)(0,0)$\\ 
		
		$m015$ & $m129(-3, 2) (0,0)$ & none\\  
		
		$m016$ & none & $m125(3, 2)(0,0)$ and $m125(-2, 3)(0,0)$\\ 
		
		$m017$ & $m129(3, 2) (0,0)$ & none \\  
		
		$m019$ & none & $m125(-3, 2)(0,0)$ and $m125(2, 3)(0,0)$ \\ 
		
		$m022$ & $m129(5, 1)  (0,0)$ & none \\  
		
		$m023$ & $m129(-5, 1)  (0,0)$ & none \\  
		
		$m026$ &none & $m125(4, 1)(0,0)$ and $m125(-1, 4)(0,0)$\\ 
		\bottomrule
	\end{tabular}
\end{table}

\subsection{Homology classes of minimisers}
\label{sec:Homology classes of minimisers}

\begin{example}
A minimiser does not need to be nullhomologous. Indeed, the unique minimiser for $L(10,3)$ is given by $m003$; see \Cref{tab:complete}. Since $m003(1,0)$ is diffeomorphic to $L(10,3)$, the minimiser $m003$ represents the complement of a knot $K$ in $L(10,3)$. We claim that $K$ is not nullhomologous. For that we observe that the first homology of $m003$ is isomorphic to $\Z\oplus\Z_5$ and thus $K$ must represent an element of order two in $H^1(L(10,3))$. More concretely, we can describe $K$ as the blue knot in the surgery diagram of $L(10,3)$ along the red unknot $U$ shown in \Cref{fig:m003}. From that diagram we compute the linking number of $K$ and $U$ to be $\pm5$ and thus $K$ represents $5$ in $\Z_{10}\cong H_1(L(10,3))$.
\end{example}

\begin{figure}[htb]
	\centering
	\includegraphics[width=5cm]{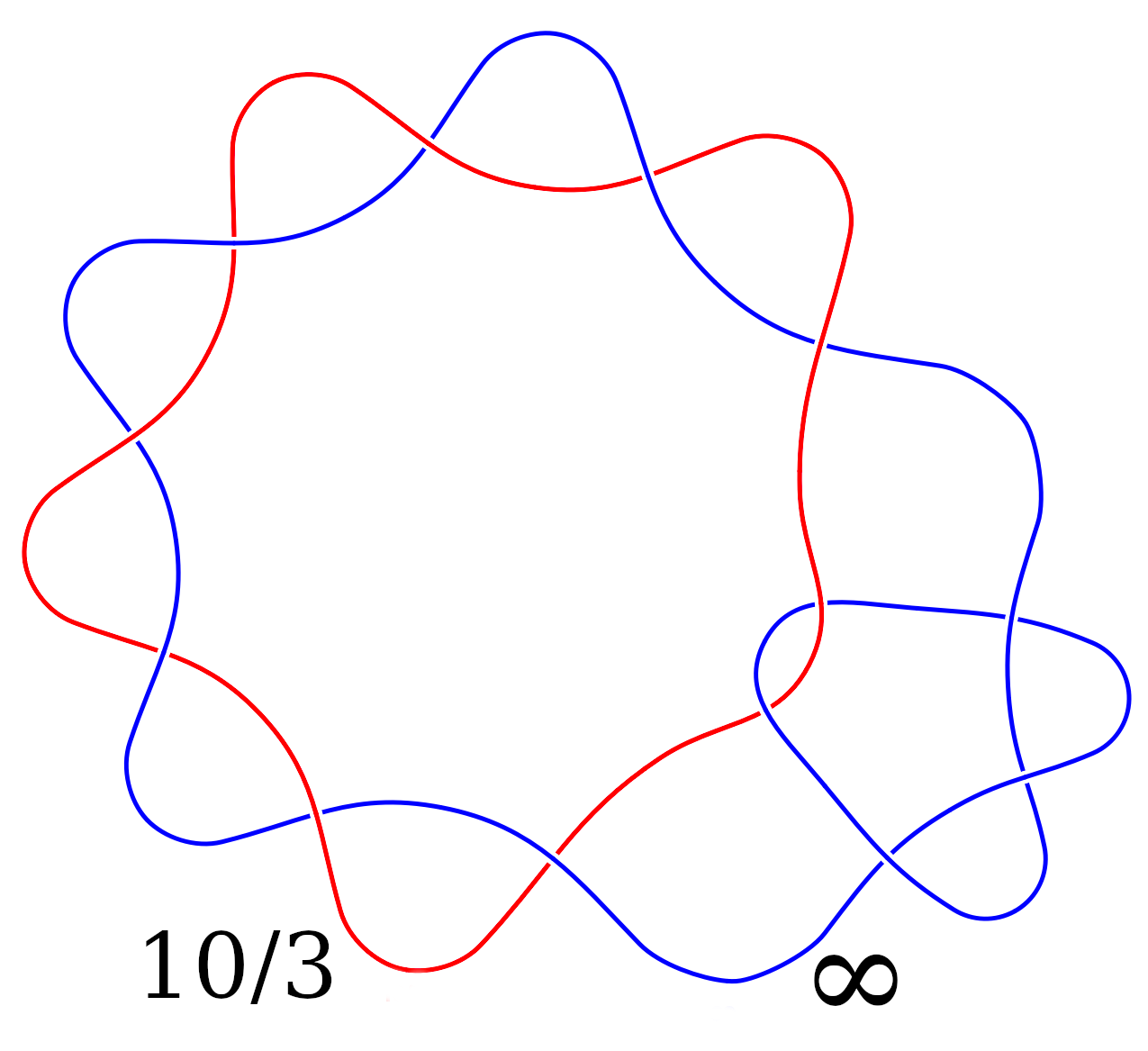}
	\caption{The surgery on the red unknot yields $L(10,3)$. We can use SnapPy to verify that the complement of the blue knot is isometric to $m003$ and thus represents the minimiser of $L(10,3)$.}
	\label{fig:m003}
\end{figure}

In the proof of \Cref{prop:census} we have listed only those knots that minimise the volume. However, among the $1$-cusped manifolds with volume at most $3.07$ there are more exceptional fillings (that do not realise) $\volt$. These are printed in \Cref{tab:non_minimizers}. Since the homologies of some of these knot complements differ from the homologies of the exteriors of the minimisers from \Cref{tab:complete}, we also obtain examples of minimisers in different homology classes.

\begin{table}[htbp]
	\caption{All other hyperbolic knots in non-hyperbolic manifolds that have volume less than $3.07$. None of these knots is a minimiser since a knot of smaller volume is already presented in \Cref{tab:complete}.}
	\label{tab:non_minimizers}
	\begin{tabular}{lll}
		knot & ambient manifold  & realisation\\
		\toprule	
		$m015$ & $S3$  & $m015(1, 0)$ \\
		$m016$ & $S3$  & $m016(1, 0)$ \\
        $m023$ & $L(3,1)$  & $m023(1, 0)$ \\
        $m019$ & $L(6,1)$  & $m019(1, 0)$ \\
		$m022$ & $L(7,2)$  & $m022(1, 1)$ \\		
		$m007$ & $SFS [S2: (2,1) (3,2) (3,-1)]$ &  $m007(2, 1)$ \\
        $m016$ & $SFS [S2: (2,1) (3,2) (5,-3)]$ &  $m016(1, 1)$ \\
		$m009$ & $SFS [S2: (2,1) (4,1) (5,-4)]$  & $m009(2, 1)$ \\
		$m023$ & $SFS [S2: (3,1) (3,1) (4,-3)]$  & $m023(-2, 1)$ \\
		$m026$ & $SFS [S2: (2,1) (3,1) (9,-7)]$  & $m026(-2, 1)$ \\
        $m016$ & $SFS [D: (2,1) (2,1)] \cup_{(0,1 | 1,1)} SFS [D: (2,1) (3,2)]$  & $m016(-2, 1)$ \\
        $m017$ & $SFS [D: (2,1) (3,1)] \cup_{(-1,1 | 0,1)} SFS [D: (2,1) (3,1)]$  & $m017(-1, 2)$ \\
        $m019$ & $SFS [D: (2,1) (2,1)] \cup_{(0,1 | 1,1)} SFS [D: (2,1) (3,1)]$  & $m019(1, 2)$ \\
		\bottomrule
	\end{tabular}
\end{table}

\begin{example}
In \Cref{tab:hom} we display the isomorphism types of the homologies of the knot exteriors from \Cref{tab:non_minimizers} whenever the homology differs from the homology of the minimiser. For example, the minimiser $m010$ of $L(6,1)$ is nullhomologous in $L(6,1)$, while the hyperbolic knot with second smallest volume $m019$ represents a generator in the homology of $L(6,1)$. 
\end{example}

\begin{example}
We also have examples of manifolds where the knot with smallest and the knot with second smallest hyperbolic volumes represent different non-trivial torsion elements in homology. For example, let $M$ be $SFS [S2: (2,1) (3,2) (3,-1)]$. Then the volume of $M$ is realised by $m003$ and the hyperbolic knot with second smallest volume is given by $m007$. Since their homologies are different they represent different non-trivial torsion elements in $H_1(M)\cong\Z_{15}$.
\end{example}

\begin{table}[htbp]
	\caption{Homologies of the minimisers compared with the homologies of the knots from \Cref{tab:non_minimizers}.}
	\label{tab:hom}
	\begin{tabular}{llllll}
		 ambient manifold  & homology & minimiser & homology & knot & homology \\
		\toprule
		$L(6,1)$  & $\Z_6$ & $m010$ & $\Z_6 + \Z$& $m019$ & $\Z$  \\
		$SFS [S2: (2,1) (3,2) (3,-1)]$  & $\Z_{15}$ & $m003$ & $\Z_5 + \Z$ & $m007$ & $\Z_3 + \Z$  \\
		$SFS [S2: (2,1) (4,1) (5,-4)]$  & $\Z_2$ & $m004$ & $\Z$ & $m009$ & $ \Z_2+ \Z$  \\
		$SFS [S2: (3,1) (3,1) (4,-3)]$  & $\Z_3$ & $m004$ & $\Z$ & $m023$ & $\Z_3 + \Z$  \\
		$SFS [S2: (2,1) (3,1) (9,-7)]$  & $\Z_3$ & $m007$ & $\Z_3 + Z$ & $m026$ & $\Z$  \\
		$SFS [D: (2,1) (2,1)] \cup_{(0,1 | 1,1)} SFS [D: (2,1) (3,2)]$  & $\Z_{20}$ & $m003$ & $\Z_5 + \Z$ & $m016$ & $\Z$  \\
		$SFS [D: (2,1) (3,1)] \cup_{(-1,1 | 0,1)} SFS [D: (2,1) (3,1)]$  & $\Z_{35}$ & $m006$ & $\Z_5 + \Z$ & $m017$ & $\Z_7 + \Z$  \\
		$SFS [D: (2,1) (2,1)] \cup_{(0,1 | 1,1)} SFS [D: (2,1) (3,1)]$  & $\Z_{28}$ & $m017$ & $\Z_7 + \Z$ & $m019$ & $\Z$  \\
		\bottomrule
	\end{tabular}
\end{table}

\subsection{Connected sums}\label{sec:connected_sum}

One natural question to ask is how topological volume behaves under connected sum. There is only one connected sum 
for which we can compute the topological volume exactly in \Cref{tab:complete}. Namely, we determine the topological volume
for $L(2,1) \# L(3,1)$ to be realised by $m010(-1, 1)$. Note that 
\begin{align*}
\volt(\; L(2,1) \# L(3,1) \; ) = \volh(m010) =& \volh(m009) \\
<&  \volh(m009) + \volh(m007) \\
=& \volt(\; L(2,1)\;) + \volt(\; L(3,1)\;).
\end{align*}
A surgery diagram of $m010$ that represents that manifold as the complement of a knot in $L(2,1) \# L(3,1)$ is shown in~\cref{fig:m010}.
\begin{question}
Let $M$ and $N$ be closed, orientable 3--manifolds. Is $\volt(\; M \# N \;) \le \volt(M)+\volt(N)$?
\end{question}

Note that Statement (1) in \Cref{cor:topinv} provides a similar, but weaker, result.

\begin{figure}[htb]
	\centering
	\includegraphics[width=6cm]{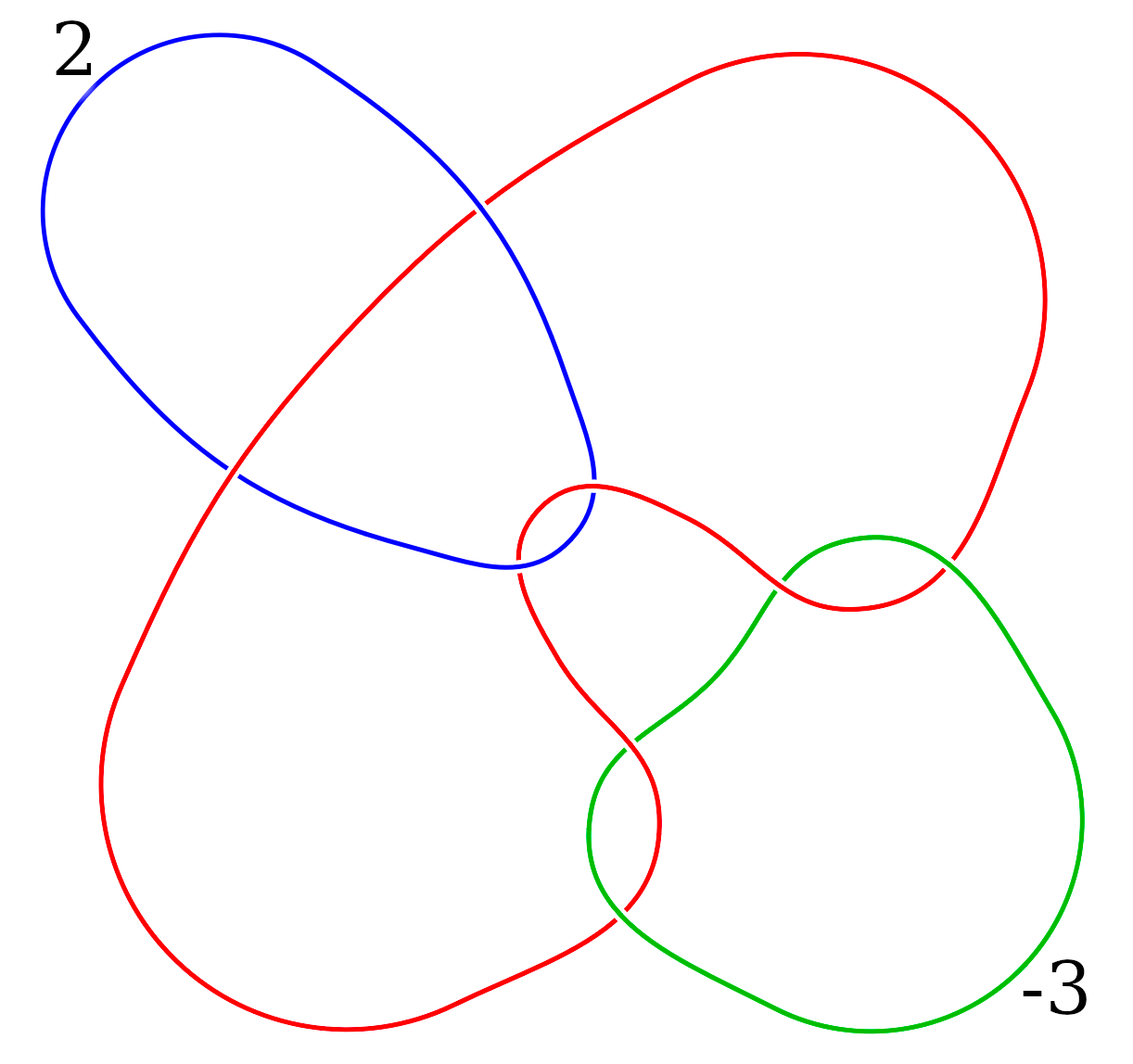}
	\caption{The manifold $m010$ seen as a knot in $L(2,1) \# L(3,1)$ minimises the volume. }
	\label{fig:m010}
\end{figure}

\begin{question}
Let $M$ be a closed, orientable 3--manifold not diffeomorphic to $S^3$. What is the limit for $n\rightarrow \infty$ of
\begin{equation*}
    \frac{\volt(\#_n M)}{n\volt(M)} \; ?
\end{equation*}
\end{question}

\section{Examples obtained from the Whitehead link and its sister}
\label{sec:WandP}

Let $W$ denote the complement of the right-handed Whitehead link with Seifert framing. This is $m129$ and $L5a1$ in the census, where the latter denotes the left-handed Whitehead link. We also let $P$ denote the $(-2,3,8)$ pretzel link complement, again, with Seifert framing. This is often called the \emph{sister manifold} of $W$ and appears as $m125$ and L13n5885 in the census of manifolds and links respectively. In the following, we use $W$ and $P$ to denote the complements of Whitehead link and sister manifold with the Seifert framing, and $m129$ and $m125$ for the Snappy census manifolds with geometric framing. 

\begin{figure}[htb]
	\centering
	\includegraphics[width=13cm]{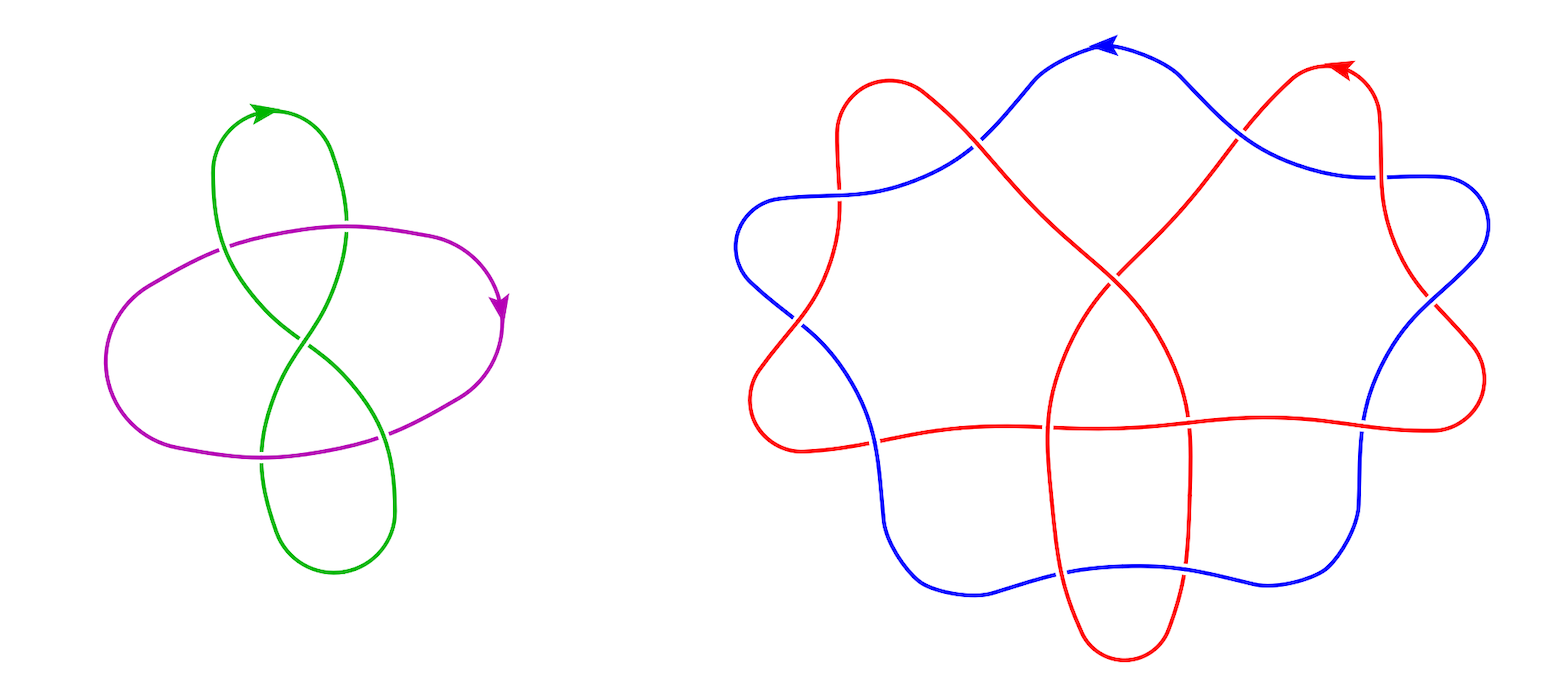}
	\caption{The Whitehead link $L5a1$ (left) and its sister link $L13n5885$ (right). The unknotted blue component is the first cusp (which we fill).}
	\label{fig:W_and_P}
\end{figure}

In SnapPy \cite{snappy}, the framing change (under filling the first cusp) is as follows:
\begin{align*}
    L5a1(p,q)(0,0)&\cong m129(p+2q,-q)(0,0)\\
    L13n5885(p,q)(0,0)&\cong m125(p-4q,p-3q)(0,0)
\end{align*}

In this section we show that, for $p^2 + q^2$ sufficiently large, $L(p,q)$ has as minimiser a filling of the Whitehead or its sister link, cf.\thinspace \Cref{prop:lens-space-limit}. 

Our computation of the smallest volume lens spaces given in \Cref{tab:complete} shows that the minimisers of all lens spaces of topological volume at most 3.07---with $p^2 + q^2$ quite small---also arise as fillings of the Whitehead or its sister link.
However, there are lens spaces for which the minimiser is not realised as a filling on either of these. One example is $L(4,1)$ (see \Cref{cor:L4-is-special}). A surgery diagram of $m034$ that displays this manifold as a knot complement of a knot in $L(4,1)$ is shown in \Cref{fig:m034}. 

\begin{question}
What is the complete list of all lens spaces that do not have minimiser a filling of the Whitehead or its sister link?
\end{question}

\begin{figure}[htb]
	\centering
	\includegraphics[width=9cm]{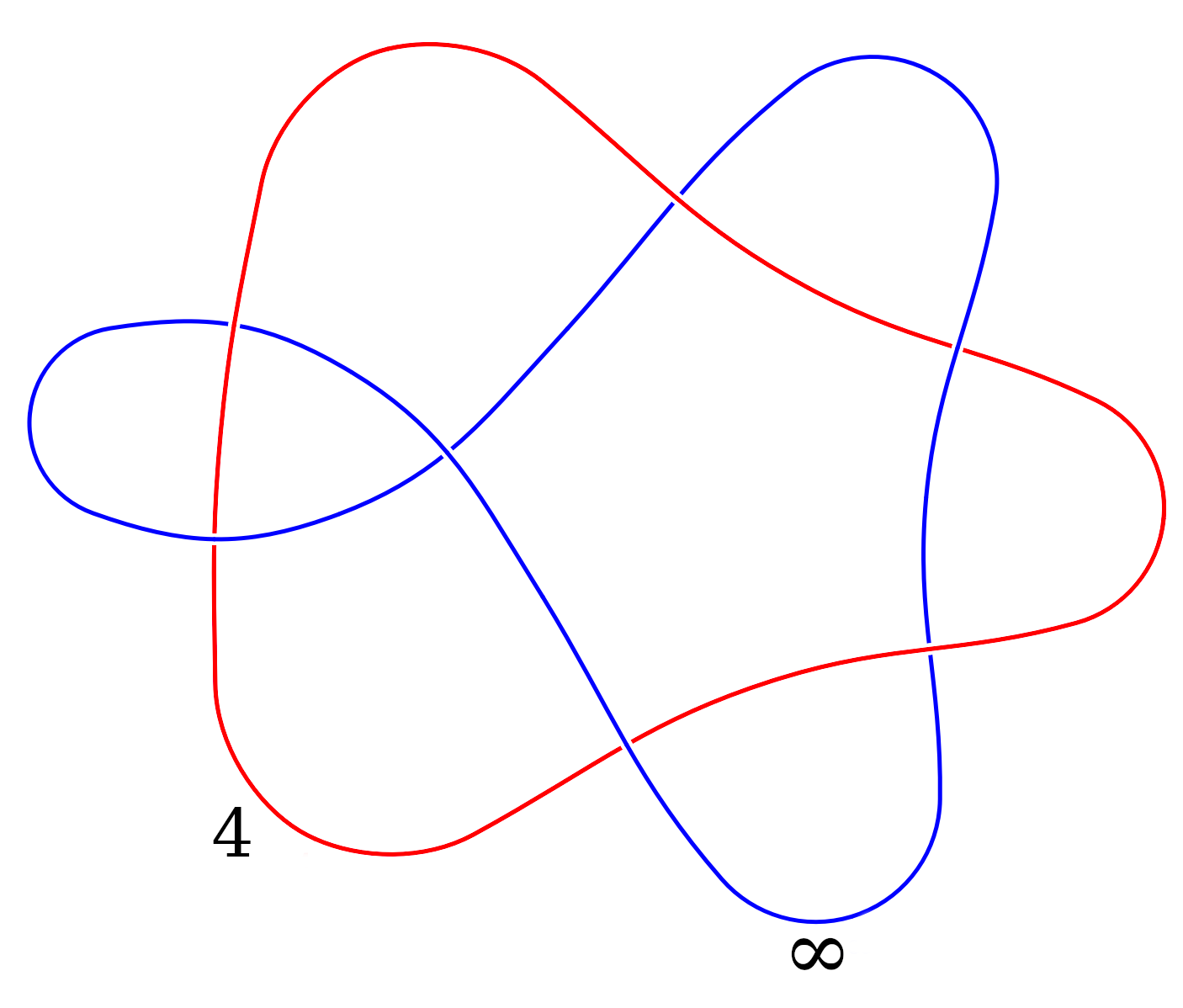}
	\caption{A surgery diagram of $m034$}
	\label{fig:m034}
\end{figure}

\subsection{Volume functions on Dehn surgery space}
\label{sec:dehnsurgeryspace}

Hodgson and Masai~\cite[\S 2]{Hodgson_Masai} give an excellent exposition of results of Neumann and Zagier~\cite{Neumann-volumes-1985} relating the decrease in volume during hyperbolic Dehn filling to the length of the surgery slope. Moreover, they analyse the corresponding functions for $P$ and $W$ in ~\cite[\S 4-5]{Hodgson_Masai}. We only state the main results required here and refer the reader to~\cite{Hodgson_Masai}  for a justification and further details.

For a cusped hyperbolic 3--manifold $M,$ let $\Delta \vol_M (p,q) = \volh(M) - \volh(M(p,q)),$ where it is understood that a fixed cusp of $M$ is filled.

We note that $(p,q)$--surgery on the first cusp of $M \in \{ P, W\}$ for $p^2+q^2$ sufficiently large gives a one-cusped hyperbolic 3--manifold $M((p,q), \infty)$. Since $M((p,q),(1,0)) \cong L(p,q),$ this gives $\volt(L(p,q)) \le \volh M((p,q), \infty)$ (note that other filling slopes may result in a homeomorphic lens space, which must be addressed in the proof of \Cref{prop:lens-space-limit}). Since $P$ and $W$ are of equal volume, we have
\[
\volh P((p,q),\infty)) - \volh W((p,q),\infty)) = \Delta \vol_W (p,q) - \Delta \vol_P (p,q).
\]

With respect to the Seifert framing on component $0$ on $W$, 
\cite[(5.14)]{Hodgson_Masai} gives:
\begin{align}
\Delta \vol_W (p,q) = 
\frac{2 \pi^2}{p^2 + 4 p q + 8 q^2}
& - \frac{\pi^4 (p^2 - 8 q^2) (p^2 + 8 p q + 8 q^2)}{3 (p^2 + 4 p q + 8 q^2)^4}\nonumber \\
&+O\Bigg( \frac{1}{(p^2 + 4 p q + 8 q^2)^3}\Bigg).\label{eq:DvolW}
\end{align}

With respect to the Seifert framing on component $0$ on $P$, \cite[(5.11)]{Hodgson_Masai} gives:
\begin{align}
\Delta \vol_P (p,q) = 
\frac{\pi^2}{(p - 4 q)^2 + (p - 3 q)^2}
&+
\frac{\pi^4 (4 p^4 - 80 p^3 q + 540 p^2 q^2 - 1520 p q^3 + 1535 q^4)}{24 ((p - 4 q)^2 + (p - 3 q)^2)^4}\nonumber \\
&+ 
O\Bigg( \frac{1}{((p - 4 q)^2 + (p - 3 q)^2)^3}\Bigg).\label{eq:DvolP}
\end{align}

\subsection{The topological volume of all but finitely many lens spaces}

Agol~\cite{agol-2cusped} showed that $W$ and $P$ are the minimal volume orientable hyperbolic 3-manifolds with two cusps. This property is key for the following:

\begin{reptheorem}{prop:lens-space-limit}
For all but finitely many homeomorphism types of lens spaces $L(p,q)$ ($p>q\geq 1$ coprime), the volume minimising link in $L(p,q)$ has complement one of $W(\frac{p'}{q'}, \infty)$ or $P(\frac{p''}{q''}, \infty)$, where the filling slopes $\frac{p'}{q'}$ and $\frac{p''}{q''}$ can be determined from an efficiently computable finite set. For $P$ this means that we fill along the unknotted component.
\end{reptheorem}

Throughout, the surgery coefficient $\frac{1}{0}$ is the filling with the infinity slope, while we do not write any slope if we leave one boundary component unfilled.

\begin{proof}
First note that, by \Cref{prop:isometry}, we can assume without loss of generality that we fill along the first cusp of both $W$ and $P$. Since the first cusp is an unknot (in $S^3$), it follows from the classification of lens spaces, that all Dehn fillings of $W$ or $P$ yielding a knot in $L(p,q)$ must have filling coefficients obtained from the following operations
\begin{align}
  (p,q) &\mapsto (p,q+np), & \forall n \in \mathbb{Z} \label{lenspace_1}\\
  (p,q) &\mapsto (p,q^*), & q^*q \equiv 1 \mod p \label{lenspace_2}\\
  (p,q) &\mapsto (-p,q), (p,-q).  \label{lenspace_3}
\end{align}
The volume functions ($\vol_W-\Delta \vol_W(p,q)$, for $\vol_W(p,q)$ from Equation \ref{eq:DvolP} and the correspondingly for $P$) derived by Hodgson and Masai in \cite{Hodgson_Masai} exclude all but finitely many fillings slopes satisfying \Cref{lenspace_1,lenspace_2,lenspace_3} from being volume minimising. 
Since these functions are strictly monotonically increasing with respect to their near-elliptical level sets (cf. \Cref{sec:dehnsurgeryspace}), this is a straightforward calculation.

With this setup, suppose to the contrary that there are infinitely many homeomorphism types of lens spaces $L(p_i,q_i)$, $p_i > q_i \geq 1$, with the property that some volume minimising link complement $M_i \subset L(p_i,q_i)$ is not homeomorphic with a Dehn filling on $P$ or $W.$ We may assume that $p_i^2 + q_i^2$ is sufficiently large, such that $W$ and $P$ have filling slopes resulting in hyperbolic knots in $L(p,q)$.

Hence for all $\frac{p'_i}{q'_i}$ obtained from $\frac{p_i}{q_i}$ by \Cref{lenspace_1,lenspace_2,lenspace_3} and with $W(\frac{p'_i}{q'_i})$, $P(\frac{p'_i}{q'_i})$ hyperbolic, we have $\volh(M_i) \le \volh(W(\frac{p'_i}{q'_i})),\volh(P(\frac{p'_i}{q'_i})) < \volh(W) = \volh(P)$. According to \cite{agol-2cusped}, $W$ and $P$ are the minimal volume orientable hyperbolic 3-manifolds with two cusps. Hence each $M_i$ has exactly one cusp.
Then $\{ \volh(M_i)\}$ is a set of volumes of hyperbolic manifolds with one cusp and bounded above by $\volh(W) = \volh(P).$ 

If this is a finite set, then there is $v_0\in \{ \volh(M_i)\}$ with the property that infinitely many of the manifolds $M_i$ have volume equal to $v_0.$ But then infinitely many lens spaces have the same volume, contradicting \Cref{pro:finiteness}.

Hence $\{ \volh(M_i)\}$ is an infinite set. But the first accumulation point for volumes of hyperbolic manifolds with one cusp is $\volh(W).$ However, a manifold with volume arbitrarily close to this value is obtained by Dehn surgery on either $P$ or $W$, a contradiction to our assumption.

This proves that all but finitely many lens spaces $L(p,q)$ have volume minimiser a Dehn filling on $P$ or $W.$ 
\end{proof}

\begin{corollary}
The sequence $(\;\volt(L(p,q))\;)_{(p,q)}$ converges, as $p^2+q^2\to\infty$, to the volume of the Whitehead link complement. \qed
\end{corollary}

\begin{corollary}
\label{cor:p-1}
The topological volume of the lens space $L(p,q)$ is realised by a knot, i.e.\ $h(L(p,q))=1$. Moreover, for $|p|$ sufficiently large, the minimiser of $L(p,1)$ is unique and the complement is a tunnel number one once-punctured torus bundle.
\end{corollary}

\begin{proof}
	Every lens space $L(p, q)$ is obtained as a $(-p/q)$-filling of the first component of (say) $P$. By performing an $n$-fold Rolfsen twist on that component we can change the surgery coefficient to $-p/(q+np)$. Thus we can assume that the filling slope is sufficiently large so that by Thurston's hyperbolic Dehn surgery theorem~\cite{Thurston_book} the other component of $P$ yields a hyperbolic knot in $L(p,q)$. But this implies that $\volt(L(p,q))<\vol(P)$. By Agol~\cite{agol-2cusped} any hyperbolic manifold with volume less than $\vol(P)$ has at most one cusp and thus the minimiser is a knot in $L(p,q)$.
  
  For the second statement, observe that $q$ must be within $\pm 1$ of a multiple of $p$ (see \Cref{lenspace_1,lenspace_2,lenspace_3}). Inserting this into the first terms of \Cref{eq:DvolW,eq:DvolP}, we can see that---for $p$ sufficiently large---the minimiser must be obtained by filling slope $(p, -1)$ on $W$ (setting $p\geq 0$ without loss): For $(p,\pm 1)$, the function $\Delta \vol_P (p,q)$ converges to $\frac{\pi^2}{2p^2}$ plus lower order terms, while $\Delta \vol_W (p,q)$ converges to $\frac{2\pi^2}{p^2}$ plus lower order terms; and we have $\Delta \vol_W (p,-1) = \frac{2\pi^2}{p^2 -4p + 1}$ while $\Delta \vol_W (p,1) = \frac{2\pi^2}{p^2 +4p + 1}$. And thus, the unique minimiser is given by an integer filling of the Whitehead link. These are known to be once-punctured torus bundles with tunnel number $1$~\cite[Theorem~1.3]{torus_bundles}.
\end{proof}

\begin{corollary}\label{cor:L4-is-special}
The lens space $L(4,1)$ has as a minimiser a one-cusped manifold of volume at most $\volh(m034)\approx 3.166$ and which is not obtained by Dehn filling $W$ or $P.$
\end{corollary}

\begin{proof}
The filling coefficients for $W$ and $P$ to obtain a knot in lens space $L(4,1)$ are precisely $(\pm 4, 2m+1)$, for $m \in \mathbb{Z}$. Of those, only $W((4,-1))$ and $P((4,1))$ are not hyperbolic: The former contains a splitting Klein bottle, the latter is a solid torus. The other cases can be verified to be hyperbolic using the $6$-theorem \cite{agol-bounds,lackenby-6} together with Snappy.

Using the volume bounds from \cite{FuterKalfagianniPurcell} for sufficiently large slopes together with direct computations using SnapPy for short slopes, we see that of the hyperbolic filling slopes, the smallest volume ones are $W((4,1))$ and $W((4,-3))$ with a volume of $\approx 3.17729$.  On the other hand, we have $\volt(L(4,1)) \le \volh(m034) \approx 3.166.$ 
\end{proof}

\begin{remark}
It is important to note that we require both the Whitehead link $W$ and its sister $P$ for \Cref{prop:lens-space-limit} to hold: The volume minimiser for $L(p,1)$, $p$ sufficiently large, is uniquely realised as filling $W(p,-1)$, as shown in \Cref{cor:p-1}. On the other hand, it can be deduced from \Cref{tab:complete} that $L(13,3)$ has minimiser $m011$ and $L(19,7)$ and $L(18,5)$ both have minimiser $m016$. But according to \Cref{tab:whichone}, these can only be obtained by Dehn fillings of $P$, but not of $W$. 
\end{remark}

We conclude this section with a question that is analogous to the third question in~\cite[\S 9]{Hodgson_Masai}.
\begin{question}
\label{q:largest}
What is the largest volume less than $\volh(W)$ of a one-cusped hyperbolic 3--manifold that does not arise from Dehn filling on $W=m129$ or $P=m125$? 
\end{question}

An answer to \Cref{q:largest} would allow us to give a quantitative version of \Cref{prop:lens-space-limit}.

\subsection{Determining the minimiser}

We are interested in an answer to the following question.

\begin{question}
  \label{q:whichone}
  Suppose $p$ and $q$ sufficiently large, and suppose $L(p,q)$ is known to have minimiser a 
  Dehn filling of $P$ or $W.$ Can we decide which of the two it is? 
\end{question}

\Cref{q:whichone} is hardest to answer when 
$\Delta \vol_W (p,q) - \Delta \vol_P (p,q) \approx 0$ for at least one (possibly 
minimising) filling slope from (the proof of) \Cref{prop:lens-space-limit}, yielding a 
knot in a given homeomorphism type of lens space. 
Using the first term Taylor series expansion for $\Delta\vol_P$ and $\Delta\vol_W$ 
from \Cref{eq:DvolW,eq:DvolP}, we have 
$\Delta \vol_W (p,q) \approx \Delta \vol_P (p,q)$ for 
\[
\frac{2 \pi^2}{p^2 + 4 p q + 8 q^2} = \frac{\pi^2}{(p - 4 q)^2 + (p - 3 q)^2}
\]
This is the case for lines 
\[ \mathbf{q}_{1} (p) = \frac{p}{42} \left ( 16 + \sqrt{130}\right)\, \text{ and }\, \mathbf{q}_{2} (p) = \frac{p}{42} \left ( 16 - \sqrt{130}\right),\]
that, for a given $p$, determine a corresponding value for $q$.
The zero set of $\Delta \vol_W (p,q) - \Delta \vol_P (p,q)$ converges to these 
lines as $p$ (and $q$) approach infinity. Note that this means that, for $p$ (and $q$) sufficiently 
large, and $\frac{p}{q}$ a fixed but arbitrary distance away from the graphs of 
$\mathbf{q}_{1,2}$, the difference $\Delta \vol_W (p,q) - \Delta \vol_P (p,q)$ can 
efficiently be bounded away from $0$ (e.g., using the first terms of their Taylor 
series expansions).

It follows that differentiating between $\Delta \vol_W (p,q)$ and $\Delta \vol_P (p,q)$
is possible except for points $(p,q)$ with the distance to the graphs of 
$\mathbf{q}_{1,2}$ arbitrarily small. Since we can efficiently compute 
$\Delta \vol_W (p,q)$ and $\Delta \vol_P (p,q)$ up to error terms in $p$ (and $q$), this 
distance must be considered with respect to the magnitude of $p$ (and $q$). In other 
words, differentiating between $\Delta \vol_W (p,q)$ and $\Delta \vol_P (p,q)$ reduces
to differentiating between those two functions for rational approximations of 
$x_1 = \frac{1}{42} \left ( 16 + \sqrt{130}\right)$ and
$x_2 = \frac{1}{42} \left ( 16 - \sqrt{130}\right)$.

For this reason, we briefly review some facts about continued fractions and refer the 
reader to \cite[Chapter 7]{ContinuedFractions} for details. Let $x \in \mathbb{R}$, and 
let $x = \langle a; a_1, \ldots , a_n, \ldots \rangle$ be its continued fraction 
expansion. We call the sequence of rational numbers $(\frac{p_i}{q_i})_{i \in \mathbb{N}}$
with $\frac{p_i}{q_i} = \langle  a; a_1, \ldots , a_i \rangle$ the {\em sequence of 
convergents} for $x$. We have $|xq_i - p_i | < | xq_{i+1} - p_{i+1} |$ for all $i \geq 0$ 
(see \cite[Theorem 7.12]{ContinuedFractions}), and 
$\left|\frac{1}{q_{i+2}} \right| < |xq_i - p_i | < \left|\frac{1}{q_{i+1}}\right|$ (see 
\cite[Theorem 7.11 and the proof of 7.12]{ContinuedFractions}). Moreover, we have the 
following statement.  

\begin{theorem}[Theorem 7.14 from \cite{ContinuedFractions}]
  \label{thm:closest}
  Let $\frac{p}{q} \in \mathbb{Q}$, $p$, $q$ co-prime, $q > 0$, and let $x \in \mathbb{R}$
  with convergents $(\frac{p_i}{q_i})_{i \in \mathbb{N}}$. If 
  $|q x - p | < |q_i x - p_i |$ for some $i$, then $q \geq q_{i+1}$. \qed
\end{theorem}

\Cref{thm:closest} asserts that the sequence of convergents for $x$ is the 
best sequence of rational numbers to gradually approximate $x$: In the statement of the theorem,
we can assume that $i$ is maximal with respect to the inequality. Hence, for an arbitrary
rational number $\frac{p}{q}$, we have $|q x - p | \geq |q_{i+1} x - p_{i+1} |$ while 
$q \geq q_{i+1}$, and the inequality is strict if we assume that $\frac{p}{q}$ is not a 
term in the sequence of convergents.

This means that the magnitude of the difference between 
$\Delta \vol_W (p,q)$ and $\Delta \vol_P (p,q)$, relative to the magnitude of error bounds, is
minimised for the convergents $(\frac{p_i}{q_i})_{i \in \mathbb{N}}$ of 
$x_1 = \frac{1}{42} \left ( 16 + \sqrt{130}\right)$ and 
$(\frac{a_j}{b_j})_{j \in \mathbb{N}}$ of 
$x_2 = \frac{1}{42} \left ( 16 - \sqrt{130}\right)$.
Thus, any procedure to bound $\Delta \vol_W$ away from $\Delta \vol_P$ for rationals
$\frac{p_i}{q_i}, \frac{a_j}{b_j} \in \mathbb{Q}$, is a procedure to bound $\Delta \vol_W$
away from $\Delta \vol_P$ for arbitrary $\frac{p}{q}$.

We devise such a procedure by determining how many terms of the Taylor series expansion of 
$\Delta \vol_W$ and $\Delta \vol_P$ we must consider, in order to certify that they have 
distinct function values at $\frac{p_i}{q_i}, \frac{a_j}{b_j} \in \mathbb{Q}$. Since we 
may exclude finitely many points, we can assume $p$ (and $q$) to be sufficiently large,
and this reduces to a study of the asymptotic behaviour of the error terms.

We have the following continued fraction expansions:
\[x_1 = \frac{1}{42} \left ( 16 + \sqrt{130}\right) = \langle 0; 1,1,1,7,7,2,7,7,2, \ldots \rangle\]
and
\[x_2 = \frac{1}{42} \left ( 16 - \sqrt{130}\right) = \langle 0; 9,7,2,7,7,2,7, \ldots \rangle.\]
From these expansions, and excluding small values of $i$ and $j$, we compute that we have 
$q_{i+1}$ (resp.\thinspace $b_{i+1}$) is at least $2$ and at most $8$ times larger than 
$q_i$ (resp.\thinspace $b_i$). Hence, the sequences of convergents gradually approach 
their limit.

By the theory of rational approximations we know that---in the worst case---integer points $(p,q)$ are at distance 
$\Theta\left(q^{-1}\right)$ from lines $\mathbf{q}_{1,2}$. First note that, for 
$p$ (and $q$) large enough and near $\mathbf{q}_{1,2}$, we have $p \in \Theta(q)$. Moreover, the
difference between the first-order terms of $\Delta \vol_W$ and $\Delta \vol_P$,
\[D(p,q) = \frac{2 \pi^2}{p^2 + 4 p q + 8 q^2} - \frac{\pi^2}{(p - 4 q)^2 + (p - 3 q)^2},\] 
at unit distance from $\mathbf{q}_{1,2}$ is of magnitude $\Theta\left(q^{-2}\right)$. 
It follows that for the convergents $\frac{p_i}{q_i}, \frac{a_j}{b_j} \in \mathbb{Q}$, that
are at distance $\Theta\left(q^{-1}\right)$ from $\mathbf{q}_{1,2}$, this amounts to 
a difference $D(p_i,q_i) \in \Theta\left(q_i^{-4}\right)$ (and, identically, 
$D(a_j,b_j) \in \Theta\left(b_j^{-4}\right)$), which can possibly be counteracted 
by the second term of the Taylor series expansion of $\Delta \vol_W$ and $\Delta \vol_P$
(which is also of order $\Theta\left(q^{-4}\right)$ near $\mathbf{q}_{1,2}$), 
but---asymptotically---not by the third term. In particular, this argument verifies the
following statement:

\begin{proposition}
  \label{prop:twoterms}
  Let $(p,q)$ be sufficiently large. Then $\Delta \vol_W (p,q) > \Delta \vol_P (p,q)$ if and only if
  \begin{align*}
   &\frac{2 \pi^2}{p^2 + 4 p q + 8 q^2} - \frac{\pi^4 (p^2 - 8 q^2) (p^2 + 8 p q + 8 q^2)}{3 (p^2 + 4 p q + 8 q^2)^4} \\
    >\qquad  &\frac{\pi^2}{(p - 4 q)^2 + (p - 3 q)^2} +
\frac{\pi^4 (4 p^4 - 80 p^3 q + 540 p^2 q^2 - 1520 p q^3 + 1535 q^4)}{24 ((p - 4 q)^2 + (p - 3 q)^2)^4}.
\end{align*} 
\end{proposition}

It follows that, if an oracle tells us lens space $L(p,q)$ has as volume minimiser a Dehn 
filling of either $W$ or $P$, we have the following procedure to determine which one it 
is: for each of the finitely many filling slopes for $(p,q)$ given in (the proof of) 
\Cref{prop:lens-space-limit}, evaluate the first two terms in the Taylor series expansions
of $\Delta \vol_W (p,q)$ and $\Delta \vol_P (p,q)$ and check the inequality in 
\Cref{prop:twoterms}. This is a deterministic, exact and efficient algorithm to check if 
$P$ or $W$ provides the minimiser, provided that one of them does.

\bibliographystyle{alpha}
\bibliography{bib}

\address{Marc Kegel\\Universidad de Sevilla, Dpto.\ de Algebra,
	Avda.\ Reina Mercedes s/n,
	41012 Sevilla, Spain\\
	{kegelmarc87@gmail.com\\-----}}

\address{Arunima Ray\\School of Mathematics and Statistics, The University of Melbourne, VIC 3010, Australia\\{aru.ray@unimelb.edu.au\\-----}}

\address{Jonathan Spreer\\School of Mathematics and Statistics F07, The University of Sydney, NSW 2006 Australia\\{jonathan.spreer@sydney.edu.au\\-----}}

\address{Em Thompson\\School of Mathematics, Monash University, Vic 3800 Australia\\{em.thompson.maths@gmail.com\\-----}}

\address{Stephan Tillmann\\School of Mathematics and Statistics F07, The University of Sydney, NSW 2006 Australia\\{stephan.tillmann@sydney.edu.au}}

\Addresses

\newpage

\appendix
\section{Using tetrahedron shapes to prove equivalence of volumes}\label{app:shapes}

In this appendix we verify our assertions regarding the equality of volumes of certain 3--manifolds, as claimed in the proof of \cref{prop:census}. 
SnapPy \cite{snappy} gives gluing information for triangulations of each 
manifold that determine the unique shapes of tetrahedra corresponding to the complete hyperbolic structure. In each of the two cases below, we show that the manifolds hypothesised to have the 
same volume are built from geometrically equivalent tetrahedra, therefore 
proving the claimed equality of volumes.

\subsection{Equivalence of volumes: m006 and m007} 

SnapPy gives the following output for

 \texttt{A=Manifold(`m006')}
and \texttt{B=Manifold(`m007')}:

\begin{verbatim}

In: A.gluing_equations(form=`rect')
Out: 
[([-1, 1, 1], [2, -2, -2], 1),
 ([0, -1, -1], [-2, 1, 1], 1),
 ([1, 0, 0], [0, 1, 1], 1),
 ([-1, 0, 0], [0, -2, 0], 1),
 ([0, 1, -1], [0, 0, 0], 1)]

In: B.gluing_equations(form=`rect')
Out: 
[([1, 2, 1], [0, 0, 0], 1),
 ([-1, -2, -1], [-1, 1, -1], 1),
 ([0, 0, 0], [1, -1, 1], 1),
 ([0, -1, -1], [-1, 0, 1], -1),
 ([1, 1, 0], [1, 0, -1], -1)]
 
\end{verbatim}

Note that $(\left[a,b,c\right],\left[d,e,f\right],g)$ translates to the equation 
\[ z_1^a (1-z_1)^d z_2^b (1-z_2)^e z_3^c (1-z_3)^f = g.\]

For clarity, we use $z$ variables for m006 and $w$ variables for m007.

Using \textit{Sage}~\cite{sage}, we find that both systems of equations have multiple solutions, but only one solution consisting of complex numbers with positive imaginary part, which is the one we are interested in. The approximate solutions are $z_1=0.7733+1.4677i, z_2=z_3=0.3352+0.4011i$, and $w_1=w_3=-0.1027+0.6654i, w_2=0.2266+1.4677i$. As such, we expect the exact values to satisfy $w_2=z_1-1, w_1=z_2/(z_2-1)$ and $w_3=z_3/(z_3-1)$. To confirm this, we compute a Groebner basis from the union of the two sets of gluing equations along with the conjectured relations between $w$ and $z$ variables. The output is
\begin{align*}
    \{z_1^3-&3 z_1^2+5 z_1-4,\\
    -&z_1^2+z_1+2 z_2-3,\\
    -&z_1^2+z_1+2 z_3-3,\\
    &w_3-z_1^2+2 z_1-3,\\
    &w_2+z_1-1,\\
    &w_1-z_1^2+2 z_1-3\}
\end{align*}

Since the Gr\"obner basis does not contain $1$, the corresponding ideal is proper and hence the system of equations admits a solution. In particular, the first polynomial in the Gr\"obner basis has $z_1\approx 0.7733+1.4677i$ as a root, and the remaining polynomials allow us to express the other variables in terms of $z_1$.

This confirms that the geometry of the tetrahedra used to triangulate the manifolds $m006$ and $m007$ are identical, and the manifolds do indeed have the same hyperbolic volume.

\subsection{Equivalence of volumes: m015, m016 and m017} 

SnapPy gives the following output for 

\texttt{C=Manifold(`m015')}, \texttt{D=Manifold(`m016')} and \texttt{E=Manifold(`m017')}:

\begin{verbatim}

In: C.gluing_equations(form=`rect')
Out: 
[([1, 1, 1], [-1, 1, -1], -1),
 ([0, 0, 0], [1, -2, 1], 1),
 ([-1, -1, -1], [0, 1, 0], -1),
 ([0, 0, 0], [-1, 1, 0], 1),
 ([3, 1, -1], [0, -1, 0], -1)]

In: D.gluing_equations(form=`rect')
Out: 
[([1, 0, 1], [-1, 1, 0], -1),
 ([-1, -1, 0], [2, -1, -2], -1),
 ([0, 1, -1], [-1, 0, 2], 1),
 ([-2, 0, -1], [2, 0, 0], 1),
 ([0, 0, 0], [-2, 2, 2], 1)]

In: E.gluing_equations(form=`rect')
Out: 
[([1, 0, -2], [-1, 1, 1], -1),
 ([0, -1, 2], [1, 0, 0], 1),
 ([-1, 1, 0], [0, -1, -1], -1),
 ([-1, 0, -2], [-1, 1, 0], -1),
 ([0, -1, 0], [-1, 0, 0], 1)]
 
\end{verbatim}

Using $u,v,w$ as variables for each system, respectively, we use \textit{Sage} to find the unique solution consisting of complex numbers with positive imaginary part for each system. In approximate values, these solutions are

$u_1=u_2=u_3=0.662359 +0.56228 i$

$v_1=0.78492 +1.30714 i,v_2=v_3=0.122561 +0.744862 i$

$w_1=0.662359 +0.56228 i,w_2=w_3=0.78492 +1.30714 i$

We therefore expect to find that the exact values for all of $u_1, u_2, u_3$ and $w_1$ are identical. Additionally, we expect $v_1=w_2=w_3=\frac{1}{1-u_1}$ and $v_2=v_3=\frac{u_1-1}{u_1}$.

Again, we compute a Groebner basis for all three sets of gluing equations along with the hypothesised relations between shapes. The output is the following: 
\begin{align*}
    \{1 - &z1 + z1^3,\\
    -&z1 + z2, -z1 + z3,\\
    &w3 - z1^2,\\
    &w2 - z1^2,\\
    &w1 - z1 - z1^2, \\
    &v3 - z1 - z1^2, \\
    &v2 - z1 - z1^2,\\
    &v1 - z1\}
\end{align*}
This is confirmation that the geometry of the tetrahedra used to triangulate the manifolds $m015$, $m016$ and $m017$ are identical, and the manifolds do indeed have the same hyperbolic volume.

\end{document}